\documentclass[12pt]{article}
\usepackage{fullpage, amsmath, amssymb, amsfonts, amsthm, stmaryrd, url, hyperref, combelow, mathtools}
\usepackage[all]{xy}
\newtheorem{theorem}{Theorem}[section]
\newtheorem{lemma}[theorem]{Lemma}
\newtheorem{cor}[theorem]{Corollary}

\newtheorem{prop}[theorem]{Proposition}
\theoremstyle{definition}
\newtheorem{defn}[theorem]{Definition}
\newtheorem{hypothesis}[theorem]{Hypothesis}
\newtheorem{example}[theorem]{Example}
\newtheorem{remark}[theorem]{Remark}

\numberwithin{equation}{theorem}

\newcommand{\FF}{\mathbb{F}}
\newcommand{\Fp}{\mathbb{F}_p}

\newcommand{\LL}{\mathbb{L}}

\newcommand{\Qp}{\mathbb{Q}_p}
\newcommand{\QQ}{\mathbb{Q}}
\newcommand{\RR}{\mathbb{R}}
\newcommand{\Zp}{\mathbb{Z}_p}
\newcommand{\ZZ}{\mathbb{Z}}
\newcommand{\bA}{\mathbf{A}}
\newcommand{\bB}{\mathbf{B}}
\newcommand{\bC}{\mathbf{C}}

\newcommand{\calA}{\mathcal{A}}
\newcommand{\calB}{\mathcal{B}}
\newcommand{\calC}{\mathcal{C}}

\newcommand{\calF}{\mathcal{F}}
\newcommand{\calG}{\mathcal{G}}
\newcommand{\calH}{\mathcal{H}}

\newcommand{\calO}{\mathcal{O}}

\DeclareMathOperator{\BMod}{\mathbf{BMod}}
\DeclareMathOperator{\Cone}{Cone}

\DeclareMathOperator{\cont}{cont}
\DeclareMathOperator{\coker}{coker}
\DeclareMathOperator{\CPhi}{\mathcal{C}\Phi}

\DeclareMathOperator{\dR}{dR}

\DeclareMathOperator{\et}{\acute{e}t}
\DeclareMathOperator{\Ext}{Ext}

\DeclareMathOperator{\GL}{GL}

\DeclareMathOperator{\image}{image}

\DeclareMathOperator{\proet}{pro\acute{e}t}

\DeclareMathOperator{\Res}{Res}
\DeclareMathOperator{\rig}{rig}

\DeclareMathOperator{\Spa}{Spa}

\DeclareMathOperator{\Tor}{Tor}
\DeclareMathOperator{\Tot}{Tot}

\title{Finiteness of cohomology of local systems on rigid analytic spaces}
\author{Kiran S. Kedlaya and Ruochuan Liu}
\date{November 21, 2016}
\begin{document}

\maketitle

\begin{abstract}
We prove that the cohomology groups of an \'etale $\QQ_p$-local system on a smooth proper rigid analytic space are finite-dimensional $\QQ_p$-vector spaces, provided that the base field is either a finite extension of $\QQ_p$ or an algebraically closed nonarchimedean field containing $\QQ_p$.
This result manifests as a special case of a more general finiteness result for the higher direct images of a relative $(\varphi, \Gamma)$-module along a smooth proper morphism of rigid analytic spaces over a mixed-characterstic nonarchimedean field.
\end{abstract}

Throughout this paper,  fix a prime number $p$, and 
let $K$ be a field of characteristic $0$ which is complete with respect to a multiplicative absolute value extending the $p$-adic absolute value on $\QQ$.
We prove the following theorem.
\begin{theorem} \label{T:finiteness1}
Let $X$ be a smooth proper rigid analytic variety over $K$.
Let $V$ be an \'etale $\QQ_p$-local system on $X$. 
\begin{enumerate}
\item[(a)]
Suppose that $K$ is algebraically closed.
Then the (continuous) cohomology groups
$H^i(X, V)$ are finite-dimensional $\QQ_p$-vector spaces for all $i \geq 0$
and vanish for $i > 2 \dim(X)$; moreover, they remain invariant under base extension from $K$ to a larger algebraically closed field.
\item[(b)]
Suppose that $K$ is a finite extension of $\QQ_p$.
Then the cohomology groups
$H^i(X, V)$ are finite-dimensional $\QQ_p$-vector spaces for all $i \geq 0$
and vanish for $i > 2 \dim(X) + 2$.
\end{enumerate}
\end{theorem}
In Theorem~\ref{T:finiteness1}, part (b) may be deduced from (a) plus Tate's local duality theorem,
which implies the special case of (b) where $X$ is a point; however, our proof of (b) will not explicitly invoke Tate's theorem.
(See Theorem~\ref{T:main finiteness etale} for the proof of (b) and Theorem~\ref{T:etale cohomology over point} for the proof of (a).)

The case of Theorem~\ref{T:finiteness1} where $V$ is the trivial local system is a consequence of Scholze's results on $p$-adic comparison isomorphisms, e.g., see \cite[Corollary~1.8]{scholze2}. To a first approximation, the proof of Theorem~\ref{T:finiteness1} uses the same basic ingredients as in \cite{scholze2}, namely the \emph{pro-\'etale topology} on $X$, the presence of \emph{perfectoid spaces} in the pro-\'etale site, the \emph{perfectoid (tilting) correspondence} between these spaces and certain spaces of characteristic $p$, and the \emph{Artin-Schreier exact sequence} in the geometric setting. However, the approach of \cite{scholze2} is ultimately limited to $\QQ_p$-local systems which arise by isogeny from $\ZZ_p$-local systems; this is sufficient for local systems occurring in algebraic geometry, but in the analytic category one acquires additional local systems that do not admit a global reduction of structure from $\GL_n(\QQ_p)$ to $\GL_n(\ZZ_p)$. The first example of this is the case where $X$ is an elliptic curve with multiplicative reduction, in which case its \'etale fundamental group admits a representation into $\QQ_p^\times$ with image $p^\ZZ$ coming from the Tate uniformization.

To deal with arbitrary $\QQ_p$-local systems, we translate the theorem into a statement about \emph{relative $(\varphi, \Gamma)$-modules} on rigid analytic spaces, as developed in \cite{part1, part2}. 
Briefly put, this amounts to viewing $V$ as a module over the constant sheaf $\underline{\QQ_p}$ on the pro-\'etale site, then tensoring  with a certain \emph{$p$-adic period sheaf}
$\tilde{\bC}$ equipped with a \emph{Frobenius endomorphism $\varphi$},
to obtain a locally free sheaf $\calF$ of $\tilde{\bC}$-modules equipped with a semilinear $\varphi$-action (i.e., a \emph{relative $(\varphi, \Gamma)$-module}).
The resulting sheaf $\calF$ then sits in an Artin-Schreier sequence
\[
0 \to V \to \calF \stackrel{\varphi-1}{\to} \calF \to 0,
\]
so we may compute the cohomology of $V$ as the hypercohomology of the complex 
$\calF \stackrel{\varphi-1}{\to} \calF$. (This is analogous to Herr's computation of Galois cohomology in terms of $(\varphi, \Gamma)$-modules \cite{herr1, herr2}, except that the role of $\Gamma$ has been absorbed by the pro-\'etale site.)

In that context, it is natural to state and prove some more general results. For example, in the context of Theorem~\ref{T:finiteness1}, $\QQ_p$-local systems give rise to relative $(\varphi, \Gamma)$-modules satisfying a certain local condition (that of being \emph{\'etale}), but the proof of finite-dimensionality applies uniformly to arbitrary
relative $(\varphi, \Gamma)$-modules; it moreover can be generalized to a relative statement about higher direct images along a smooth proper morphism of rigid spaces
(Theorem~\ref{T:direct image}). 
The more general relative $(\varphi, \Gamma)$-modules are of some interest in their own right, in part because they may be interpreted as vector bundles on certain schemes or adic spaces related to the \emph{curves in $p$-adic Hodge theory} introduced by Fargues and Fontaine \cite{fargues-fontaine} and further studied in \cite{part1}. In particular, moduli spaces of such bundles are thought to give rise to the local Langlands correspondence in a fashion analogous to that seen in the case of function fields of positive characteristic; this imparts some urgency to the study of cohomology of such spaces.

In addition to constructing higher direct images, we establish their compatibility with base change (Theorem~\ref{T:main finiteness etale}) and give a relative version of the \'etale-de Rham comparison isomorphism for rigid analytic spaces constructed by Scholze \cite{scholze2} (Theorem~\ref{T:comparison isomorphism}). Note that the latter applies to $\QQ_p$-local systems, whereas the methods of \cite{scholze2} are only capable of handling $\ZZ_p$-local systems; the latter tend to arise from algebraic-geometric constructions, whereas the former typically arise from genuinely analytic constructions (e.g., the Tate uniformization of an elliptic curve, or the Gross-Hopkins period morphism).

Even with the present work, some properties of relative cohomology of $(\varphi, \Gamma)$-modules remain to be developed, notably cohomology with compact supports and sheaf duality. We defer these topics to a later occasion.

\subsection*{Acknowledgments}

During the writing of this paper,
Kedlaya was supported by NSF grant DMS-1501214,
UC San Diego (Stefan E. Warschawski Professorship), and a Guggenheim Fellowship. Liu was supported by NSFC-11571017. Thanks to David Hansen for helpful feedback.

\section{Completely continuous homomorphisms}
\label{subsec:finiteness}

We begin with a key ingredient for finiteness results: the method of Cartan-Serre for proving finiteness of cohomology, using the Schwartz lemma on completely continuous homomorphisms of Banach spaces.
We follow the treatment given by Kiehl in his proof of preservation of coherence along proper morphisms of rigid analytic spaces \cite{kiehl-finiteness}, except that we make some crucial technical modifications to avoid unwanted dependence on noetherian hypotheses (Remark~\ref{R:Kiehl def}).

\begin{hypothesis}
Throughout \S\ref{subsec:finiteness}, fix a
Banach ring $A$ according to the convention of
\cite{part1, part2}, i.e., a commutative ring complete with respect to a submultiplicative nonarchimedean absolute value and containing a topologically nilpotent unit.
We work in the category $\BMod_A$ whose objects are Banach modules over $A$ and whose morphisms are bounded $A$-linear homomorphisms;
note that these homomorphisms satisfy the open mapping theorem
\cite[Theorem~2.2.8]{part1}, so in particular every surjective morphism is strict.
\end{hypothesis}

\begin{remark} \label{R:noetherian}
If $A$ is noetherian, then any finite $A$-module is a Banach module \cite[Remark~2.2.11]{part1}; in particular, any morphism in $\BMod_A$ with finitely generated image is strict (by the open mapping theorem). By contrast, if $A$ is not noetherian, every finite $A$-module $M$ has a unique \emph{natural topology} (induced by any $A$-linear surjection from a finite free module), but need not be a Banach module for this topology;
this occurs if and only if $M$ is Hausdorff (by the open mapping theorem; see
\cite[Remark~1.2.6]{part2}).
In particular, any finite $A$-submodule of a Banach module over $A$ is itself complete for  its natural topology, although not necessarily for the subspace topology.

One important class of finite $A$-modules which are complete for the natural topology is 
the class of \emph{pseudocoherent} $A$-modules; see Definition~\ref{D:pseudocoherent}.
\end{remark}

\begin{defn}  \label{D:completely continuous}
Let $f: M \to N$ be a morphism in $\BMod_A$. We say that $f$ is \emph{completely continuous}
if there exists a sequence of finite $A$-submodules $N_i$ of $N$ such that the operator norms of the compositions $M \to N \to N/N_i$, for some fixed norms on $M$ and $N$ and the quotient seminorm on $N/N_i$, converge to 0. (This is slightly weaker than the usual definition; see Remark~\ref{R:Kiehl def}.)
\end{defn}

A basic fact about completely continuous morphisms is the Schwartz lemma (compare \cite[Satz~1.5]{kiehl-finiteness}). We will need a slightly different result to control cohomology
(see Lemma~\ref{L:Schwartz2} below), but giving the proof of this statement provides a useful illustration of technical details to appear later.
\begin{lemma} \label{L:Schwartz}
Let $f,g: M \to N$ be morphisms in $\BMod_A$ such that $f$ is surjective
and $g$ is completely continuous. Then $\coker(f+g)$ is a finite $A$-module.
(Note that this does not immediately imply that $f+g$ is strict unless $A$ is noetherian; see Remark~\ref{R:noetherian}.)
\end{lemma}
\begin{proof}
Choose a finite free $A$-module $U$ admitting an $A$-linear morphism $\pi': U \to N$ such that $U \stackrel{\pi'}{\to} N \to \coker(f)$ is surjective.
Fix norms on $M,N,U$ and a real number $t>1$.
By the open mapping theorem, $f \oplus \pi': M \oplus U \to N$ admits a bounded set-theoretic section $s \oplus s': N \to M \oplus U$; let $c$ be the operator norm of $s$.
By Remark~\ref{R:noetherian}, 
Since $g$ is completely continuous, we can find a finite $A$-submodule $V$ of $N$ such that for the quotient seminorm on $N/V$, the
composition of $g$ with the projection $\pi'': N \to N/V$ has operator norm at most $t^{-2} c^{-1}$. 
The map $\pi$ admits a set-theoretic section $s'': N/V \to N$ which is bounded of norm at most $t$.

For $n \in N$, define $m_i \in M, u_i \in U, v_i \in V$ for $i=0,1,\dots$ recursively as follows. First set $n_0 := n$. Given $n_i$, put
\[
(m_i, u_i) := (s \oplus s')(n_i), \qquad
v_i := ((1 - s'' \circ \pi'') \circ g)(m_i), \qquad
n_{i+1} := (s'' \circ \pi'' \circ g)(m_i).
\]
We then have $\left|n_{i+1}\right| \leq t^{-1} \left| n_i \right|$, 
so the $n_i$ converge to 0; since all maps in the construction are bounded, the $m_i, u_i, v_i$ also converge to 0. By construction,
\[
(f+g)(m_i) + \pi(u_i) = n_i + g(m_i) = n_i + v_i - n_{i+1};
\]
if we sum this relation over $i$ and set $m := \sum_{i=0}^\infty m_i$, 
$u = \sum_{i=0}^\infty u_i$, $v = \sum_{v=0}^\infty v_i$, we have
\[
(f+g)(m) + \pi(u) = n + v.
\]
Morover, $v \in N$ belongs to $V$ because the latter is complete by Remark~\ref{R:noetherian}.
It follows that $V \to \coker(f+g)$ is surjective,
so $\coker(f+g)$ is a finite $A$-module. 
\end{proof}

\begin{remark} \label{R:Schwartz trick}
The usual statement of Lemma~\ref{L:Schwartz} requires $f$ to be surjective, rather than allowing the cokernel to be nonzero but finitely generated; this is the level of generality needed in most applications, including ours. In fact, the formally stronger assertion is itself an immediate consequence of the surjective case, as one may replace the original maps $f, g: M \to N$ by maps $f \oplus \pi, g \oplus 0: M \oplus U \to N$
where $\pi: U \to N$ is an $A$-linear morphism from a finite free $A$-module $U$ such that the composition $U \stackrel{\pi}{\to} N \to \coker(f)$ is surjective.

It is therefore not an obvious design choice to proceed as we have, by bundling the finite cokernel into the proof of Lemma~\ref{L:Schwartz}. The justification for this choice is to illustrate the fact that the constant $c$ used to choose the subspace $V$ depends only on $s$, not on $s'$. In the context of Lemma~\ref{L:Schwartz} this distinction is immaterial, but we use a similar argument to crucial effect in the proof of Lemma~\ref{L:Schwartz2}, and it is the latter statement that is needed for cohomological applications via Lemma~\ref{L:Cartan-Serre}.
\end{remark}

\begin{remark} \label{R:composition}
Let $f: M \to N$ be a morphism in $\BMod_A$. If $f$ is completely continuous, then it is obvious that the precomposition of $f$ with any morphism $g: M' \to M$ in $\BMod_A$ is again completely continuous. Conversely, if $g$ is a (necessarily strict) surjection and 
$f \circ g$ is completely continuous, then so is $f$.

In the other direction, if $f$ is completely continuous, then the postcomposition of $f$ with a morphism $g: N \to N'$ in $\BMod_A$ is again completely continuous. However, if $g$ is a strict inclusion and $g \circ f$ is completely continuous, then it is not obvious that $f$ is completely continuous unless $g$ splits in $\BMod_A$, as otherwise it is not clear how to ``project'' finite $A$-submodules of $N'$ to finite $A$-submodules of $N$ in a sufficiently uniform way.
\end{remark}

\begin{remark} \label{R:fingen cc}
Let $R \to S$ be a bounded morphism of Banach algebras over $A$ which is completely continuous as a morphism in $\BMod_A$. Let $M$ be a finite Banach module over $R$ such that $M \otimes_R S$ is a Banach module over $S$. For $F \to M$ an $R$-linear surjection from a finite free $R$-module, it is obvious that $F \to F \otimes_R S$ is completely continuous, as then is $F \to F \otimes_R S \to M \otimes_R S$. Refactoring this composition as $F \to M \to M \otimes_R S$, we may remove the surjection $F \to M$ as per Remark~\ref{R:composition} to deduce that $M \to M \otimes_R S$ is again completely continuous.
\end{remark}

The crucial example of Remark~\ref{R:fingen cc} in \cite{kiehl-finiteness} involves the following class of morphisms.
\begin{example}  \label{exa:Tate completely continuous}
For $i=1,\dots,n$, choose real numbers $r_i, r'_i, s_i, s'_i$ such that $0 < s_i < s'_i \leq r'_i < r_i$. Then the natural inclusions
\begin{gather*}
A \{T_1/r_1,\dots,T_n/r_n\} \to A\{T_1/r'_1,\dots,T_n/r'_n\} \\
A \{s_1/T_1, T_1/r_1,\dots,s_n/T_n, T_n/r_n\} \to A\{s'_1/T_1, T_1/r'_1,\dots,s'_n/T_n, T_n/r'_n\}
\end{gather*}
are strictly completely continuous, taking the submodules generated by the initial segments of an enumeration of monomials.
\end{example}

In light of the previous remark, we formulate the following variant of Lemma~\ref{L:Schwartz}.
\begin{lemma} \label{L:Schwartz2}
Let $f,g: M \to N, h: N \to N'$ be morphisms in $\BMod_A$ such that
$\coker(f)$ is a finite $A$-module, $h \circ g$ is completely continuous, and $h$ is a strict inclusion. Then $\coker(f+g)$ is contained in a finite $A$-module.
\end{lemma}
\begin{proof}
To simplify notation slightly, we proceed as in Remark~\ref{R:Schwartz trick} to formally reduce to the case where $f$ is surjective.
Fix norms on $M,N,N'$ and a real number $t>1$.
By the open mapping theorem, $f$ admits a bounded set-theoretic section $s: N \to M$; let $c$ be the operator norm of $s$.
By Remark~\ref{R:noetherian}, 
Since $h \circ g$ is completely continuous, we can find a finite $A$-submodule $V$ of $N'$ such that for the quotient seminorm on $N'/V$, the
composition of $h \circ g$ with the projection $\pi'': N' \to N'/V$ has operator norm at most $t^{-3} c^{-1}$. 
The map $\pi''$ admits a set-theoretic section $s'': N'/V \to N'$ which is bounded of norm at most $t$.

Since $V$ is a finite $A$-module, its image in $N'/N$ is again a finite $A$-module,
and hence a Banach module (Remark~\ref{R:noetherian}); let $N''$ be the inverse image of this module in $N'$, which is then also a Banach module.
Choose a finite free $A$-module $U$ and an $A$-linear morphism $\pi: U \to N''$
such that $U \stackrel{\pi}{\to} N'' \to N''/N$ is surjective.
Choose a set-theoretic section $s'''$ of the projection $\pi''': N'' \to N''/N$ of operator norm at most $t$ taking 0 to 0. By composing $1 - s''' \circ \pi''': N'' \to N$ with the previously chosen section $s: N \to M$, we obtain a set-theoretic map $s: N'' \to M$
of operator norm at most $ct$ extending $s$, which can be complemented to a bounded set-theoretic section $s \oplus s': N'' \to M \oplus U$ of the surjective map $f \oplus \pi$. (We cannot control the norm of $s'$, but as per Remark~\ref{R:Schwartz trick} this will have no effect on the sequel.)

We now emulate the proof of Lemma~\ref{L:Schwartz}.
For $n \in N''$, define $m_i \in M, u_i \in U, v_i \in V$ for $i=0,1,\dots$ recursively as follows. First set $n_0 := n$. Given $n_i$, put
\[
(m_i, u_i) := (s \oplus s')(n_i), \qquad
v_i := ((1 - s'' \circ \pi'') \circ g)(m_i), \qquad
n_{i+1} := (s'' \circ \pi'' \circ g)(m_i).
\]
We then have $\left|n_{i+1}\right| \leq t^{-1} \left| n_i \right|$, 
so the $n_i$ converge to 0; since all maps in the construction are bounded, the $m_i, u_i, v_i$ also converge to 0. By construction,
\[
(f+g)(m_i) + \pi(u_i) = n_i + g(m_i) = n_i + v_i - n_{i+1};
\]
if we sum this relation over $i$ and set $m := \sum_{i=0}^\infty m_i$, 
$u = \sum_{i=0}^\infty u_i$, $v = \sum_{v=0}^\infty v_i$, we have
\[
(f+g)(m) + \pi(u) = n + v.
\]
Moreover, $v \in N'$ belongs to $V$ because the latter is complete by Remark~\ref{R:noetherian}.
It follows that the map $U \oplus V \to \pi(U) \oplus V \to \coker(f+g: M \to N'')$ is surjective, and hence $\coker(f+g: M \to N'')$ is a finite $A$-module.
From the exact sequence
\[
0 \to \coker(f+g: M \to N) \to \coker(f+g: M \to N'') \to N''/N \to 0,
\]
we deduce the claim.
\end{proof}

The following consequence of Lemma~\ref{L:Schwartz2} is essentially \cite[Satz~2.5]{kiehl-finiteness}, except with weaker hypotheses (see Remark~\ref{R:Kiehl def}).
\begin{lemma}[after Cartan-Serre] \label{L:Cartan-Serre}
Let $f^\bullet: C_1^{\bullet} \to C_2^{\bullet}$ be a morphism of complexes in $\BMod_A$. 
Suppose that for some $i$, the following conditions hold.
\begin{enumerate}
\item[(a)]
The morphism $f^i$ is completely continuous.
\item[(b)]
The map $H^i(C_1) \to H^i(C_2)$ induced by $f^\bullet$ has cokernel which is a finite $A$-module.
\end{enumerate}
Then the group $H^i(C_2)$ is contained in a finite $A$-module. In particular, if $f^\bullet$ is a quasi-isomorphism, then $H^i(C_1)$ is contained in a finite $A$-module.
\end{lemma}
\begin{proof}
Put $Z^i_j := \ker(d^i_j: C^i_j \to C^{i+1}_j)$.
By (a) and (b) and Remark~\ref{R:composition}, the map
\[
Z^i_1 \oplus C_2^{i-1} \to Z^i_2, \qquad 
(a,b) \mapsto -f^i(a)
\]
is the composition of a completely continuous morphism with a strict inclusion,
while the map
\[
Z^i_1 \oplus C_2^{i-1} \to Z^i_2, \qquad 
(a,b) \mapsto f^i(a) + d_2^{i-1}(b)
\]
has cokernel which is a finite $A$-module.
We may identify $H^i(C_2)$ with the cokernel of the map
\[
Z^i_1 \oplus C_2^{i-1} \to Z^i_2, \qquad 
(a,b) \mapsto d_2^{i-1}(b),
\]
and by Lemma~\ref{L:Schwartz2} this cokernel is contained in a finite $A$-module.
This proves the claim.
\end{proof}

\begin{remark} \label{R:Cartan-Serre use case}
Our initial applications of Lemma~\ref{L:Cartan-Serre} will occur in the following context: we have a bounded homomorphism $R \to S$ of Banach algebras over $A$ which is completely continuous in $\BMod_A$ and a complex $C_1^\bullet$ of finite Banach modules over $R$ such that each module $C_2^\bullet = C_1^\bullet \otimes_R S$ is complete and the induced map
$C_1^\bullet \to C_2^\bullet$ is a quasi-isomorphism.
In this setting, 
condition (a) of Lemma~\ref{L:Cartan-Serre} will be satisfied thanks to
Remark~\ref{R:fingen cc}.

Our core application of Lemma~\ref{L:Cartan-Serre} will be in a somewhat more complicated setting. See Remark~\ref{R:Cartan-Serre use case2}.
\end{remark}

\begin{remark} \label{R:Kiehl def}
We say that the morphism $f: M \to N$ in $\BMod_A$ is \emph{completely continuous in the sense of Kiehl} (i.e., \emph{vollst\"andig stetig} in the sense of \cite[Definition~1.1]{kiehl-finiteness})
if it can be uniformly approximated by a sequence of morphisms $f_i: M \to N$ whose images are finite $A$-modules. This implies that $f$ is completely continuous in the sense of
Definition~\ref{D:completely continuous} by taking the submodules of $N$ to be the images of the $f_i$. The converse implication is not clear in general, but it does hold when $A$ is a nonarchimedean field:
the projection $\pi_i: N \to N/N_i$ admits an $A$-linear section $s$ whose operator norm is bounded uniformly in $i$ (e.g., because \cite[Lemma~1.3.7]{kedlaya-course} allows the operator norm to be made arbitrarily close to 1), and the sequence of maps
$f_i = (1 - s \circ \pi_i) \circ f$ then has the desired effect.
It also holds if $M$ is \emph{topologically projective}, i.e., it is a direct summand of the completion of a free $A$-module with respect to the supremum norm; in this case, one may construct the maps $f_i$ by assigning images to topological generators. The latter observation has the net effect of minimizing any substantive difference between our definition and Kiehl's definition, and also explains why various intermediate results in \cite{kiehl-finiteness} require precomposition with a surjection from a topologically free module.

However, there is no exact analogue in \cite{kiehl-finiteness} of Lemma~\ref{L:Schwartz2}.
Instead, Kiehl formulates in \cite[Definition~1.1]{kiehl-finiteness} a stronger version of the complete continuity condition, defining what it means for a morphism $f$ to be \emph{strictly completely continuous} (\emph{streng vollst\"andig stetig}) by imposing a certain uniformity condition on the maps $f_i$. The purpose of this stronger condition is to make it possible to deduce strict complete continuity after postcomposition with a strict inclusion. Unfortunately, Kiehl is only able to achieve this implication in case $A$ is an affinoid algebra over a nonarchimedean field \cite[Satz~1.4]{kiehl-finiteness},
because the argument requires a very strong noetherian property which is only verified in this case \cite[Satz~5.1]{kiehl-finiteness}: for $A_0$ a ring of definition of $A$,
every $A_0$-submodule of every finite $A_0$-module is \emph{almost finitely generated} in the sense of almost ring theory. While it is reasonable to expect this for other strongly noetherian Banach rings, such as the coordinate rings of affinoid subsets of Fargues-Fontaine curves \cite{kedlaya-noetherian}, it is immaterial for our purposes: if $A$ is noetherian, then the conclusion of Lemma~\ref{L:Cartan-Serre} promotes to the statement that $H^i(C_2)$ is a finite $A$-module.
Moreover, in our application of Lemma~\ref{L:Cartan-Serre} to the proof of Proposition~\ref{P:direct image finiteness}, the base ring $A$ will not be noetherian;
this means that we will have to use extra structure on cohomology, plus some deep finiteness results from \cite{part2}, in order to promote the conclusion (that cohomology groups are contained in finite $A$-modules) to stronger finite generation assertions.

In any case, we leave it as an easy exercise to recover the main theorem of \cite{kiehl-finiteness}, with an arguably simpler proof, using the lemmas stated above.
\end{remark}

\section{Totalizations}\label{S:totalization}

As much of this paper is concerned with matters of homological algebra, it is best to review some basic formalism in order to fix notation and conventions, especially with regard to converting multidimensional complexes into single complexes;
this allows us to state a crucial formal promotion of Lemma~\ref{L:Cartan-Serre}.
Throughout \S\ref{S:totalization}, fix an exact additive category $\calA$.

\begin{defn}
Let $C^+(\calA)$ be the category of bounded below complexes in $\calA$ with cohomological indexing, so that for $C \in C^+(\calA)$, the differential $d_C^i$ maps from $C^i$ to $C^{i+1}$; recall that a morphism $f^\bullet: C_1 \to C_2$ in $C^+(\calA)$
consists of the vertical arrows in a commutative diagram of the form
\[
\xymatrix{
\cdots \ar[r] & C_1^{i-1} \ar^{d_{C_1}^{i-1}}[r] \ar^{f^{i-1}}[d] & C_1^i \ar^{d_{C_1}^i}[r] \ar^{f^i}[d] & \cdots \\
\cdots \ar[r] & C_2^{i-1} \ar^{d_{C_2}^{i-1}}[r]  & C_2^i \ar^{d_{C_2}^i}[r] & \cdots.
}
\]
Such a morphism is \emph{null-homotopic} if there exist maps $h^i: C_1^i \to C_2^{i-1}$ such that
\begin{equation} \label{eq:homotopy}
f^i = d_{C_2}^{i-1} \circ h^i + h^{i+1} \circ d_{C_1}^i;
\end{equation}
this implies that $f$ induces zero maps $h^\bullet(C_1) \to h^\bullet(C_2)$ of cohomology groups. (Note that \eqref{eq:homotopy} alone implies that $f$ is a morphism in $C^+(\calA)$, i.e., that the diagram commutes.)
One checks that the morphisms which are null-homotopic form an ideal under composition; quotienting the morphism spaces by this ideal yields the \emph{homotopy category} $K^+(\calA)$.
\end{defn}

\begin{defn} \label{D:mapping cone}
For $f: C_1 \to C_2$ a morphism in $C^+(\calA)$, the \emph{mapping cone} is the object $\Cone(f) \in C^+(\calA)$ such that $\Cone(f)^i = C_1^i \oplus C_2^{i-1}$
and
\[
d_{\Cone(f)}^i(x, y) = (d_{C_1}^i(x), d_{C_2}^{i-1}(y) + (-1)^i f^i(x));
\]
there are obvious morphisms $C_2[-1] \to \Cone(f) \to C_1$ which compose to zero (and close up into a \emph{distinguished triangle}). The class of $\Cone(f)$ in $K^+(\calA)$ is  functorially determined by the class of $f$ in $K^+(\calA)$.

Suppose that $g: C_0 \to C_1$ is a second morphism in $C^+(\calA)$ such that the composition $C_0 \to C_1 \to C_2$ is null-homotopic as witnessed by $h^i: C_0^i \to C_2^{i-1}$. We may then construct a morphism $C_0 \to \Cone(f)$ in $C^+(\calA)$ 
mapping $C_0^i$ to $C_1^i \oplus C_2^{i-1}$ via $g^i \oplus - h^i$.
The class of $C_0 \to \Cone(f)$ in $K^+(\calA)$ is functorially determined by the class of $g$ in $K^+(\calA)$. 
\end{defn}

\begin{remark} \label{R:mapping cone func}
The second observation in Definition~\ref{D:mapping cone} may be restated as follows.
Let
\[
\xymatrix{
C_1 \ar^{f}[r] \ar^{e_1}[d] & C_2 \ar^{e_2}[d] \\
D_1 \ar^g[r] & D_2
}
\]
be a diagram in $C^+(\calA)$ which commutes in $K^+(\calA)$, and choose a homotopy 
$h^\bullet: C_1^\bullet \to D_2^{\bullet - 1}$ witnessing the vanishing of $e_2 \circ f - g \circ e_1$ in $K^+(\calA)$. Then 
the induced morphism $\Cone(f) \to \Cone(g)$ can be represented by a morphism of complexes whose component $C_1^i \oplus C_2^{i-1} \to D_1^i \oplus D_2^{i-1}$ acts as
\[
(a,b) \mapsto (e_1^i(a), e_2^{i-1}(b)+(-1)^{i+1}h^i(a)).
\]
In particular, this map depends on $h$.

This observation has the following consequence in our setting. Take $\calA = \BMod_A$ for $A$ a Banach ring. If $e_1, e_2$ consist of completely continuous morphism, it does not follow that the resulting map $\Cone(f) \to \Cone(g)$ has the same property unless we can ensure that the homotopy $h$ is itself completely continuous. The easiest way to achieve this is to ensure that in fact $h=0$, which is to say the original diagram commutes already in $C^+(\calA)$; fortunately, this is what will ultimately happen in the case of interest.
\end{remark}

\begin{defn}\label{D: totalization}
Let $C_{\bullet}:=C_0 \to \cdots \to C_n$ be a finite sequence in $C^+(\calA)$ that represents a complex in $K^+(\calA)$; that is, the compositions $C_{i-1}\to C_{i}\to C_{i+1}$ are null-homotopic for all $i$. We now repeat the following operation: take the last two terms $C_{n-1} \to C_n$ and replace them by the single term $\Cone(C_{n-1} \to C_n)$. After $n$ steps, we end up with an element of $C^+(\calA)$ whose class in $K^+(\calA)$ is uniquely determined by the image of $C_{\bullet}$ in the category of bounded complexes in $K^+(\calA)$. We call this element the \emph{totalization} of $C_{\bullet}$, and denote it by $\Tot(C_{\bullet})$.
\end{defn}

\begin{remark}\label{R:double-totalization}
In the case that the sequence $C_{\bullet}$ is indeed a complex in $C^+(\calA)$ (i.e., the compositions $C_{i-1} \to C_i \to C_{i+1}$ are actually zero, not merely null-homotopic),  the totalization of $C_{\bullet}$ is represented by the total complex of the double complex $C^{i,j}:=C_j^i$.  
\end{remark}

\begin{remark}\label{R:morphism-totalization}
For $i = 1, 2$, let $C_{i, \bullet} := C_{i,0} \to \cdots \to C_{i,n}$ be a sequence in $C^+(\calA)$ representing a complex in $K^+(\calA)$. Let $f_{\bullet}: C_{1, \bullet} \to C_{2, \bullet}$ be a family of morphisms in $C^+(\calA)$ such that for each $j$, the diagram
\[
\xymatrix{
C_{1,j-1} \ar^{f_{j-1}}[d] \ar[r] & C_{1,j} \ar^{f_j}[d] \ar[r] & C_{1,j+1} \ar^{f_{j+1}}[d] \\
C_{2,j-1} \ar[r] & C_{2,j} \ar[r] & C_{2,j+1}
}
\]
commutes in $C^+(\calA)$, not only in $K^+(\calA)$. Moreover, suppose that one can choose homotopies witnessing that each row is a complex in $K^+(\calA)$ which themselves form commutative diagrams with the $f_j$. It is straightforward to see that the family $f_\bullet$ naturally induces a morphism $\Tot(f_\bullet):\Tot(C_{1,\bullet})\to\Tot(C_{2,\bullet})$ in $C^+(\calA)$.
\end{remark}

Unfortunately, it is unclear how to formulate an analogue of Lemma~\ref{L:Cartan-Serre} which up to full homotopy equivalence. Instead, we give a more rigid statement which provides just enough flexibility for our purposes.
\begin{lemma} \label{L:Cartan-Serre homotopy}
Let $A$ be a Banach ring and 
let $n$ be a nonnegative integer.
For $i = 1, 2$, let $C_{i, \bullet} := C_{i,0} \to \cdots \to C_{i,n}$ be a sequence in $C^+(\BMod_A)$ representing a complex in $K^+(\BMod_A)$. Let $f_{\bullet}: C_{1, \bullet} \to C_{2, \bullet}$ be a family of morphisms in $C^+(\BMod_A)$.
For each $0\leq j\leq n$, let $\alpha_{i,j}: D_{i,j} \to C_{i,j}$,
$\beta_{i,j}: C_{i,j} \to D_{i,j}$ be morphisms in $C^+(\BMod_A)$ which are inverses to each other in $K^+(\BMod_A)$.
Suppose that the following conditions are satisfied.
\begin{enumerate}
\item[(a)]
For each $j$, the diagram
\[
\xymatrix{
C_{1,j-1} \ar^{f_{j-1}}[d] \ar[r] & C_{1,j} \ar^{f_j}[d] \ar[r] & C_{1,j+1} \ar^{f_{j+1}}[d] \\
C_{2,j-1} \ar[r] & C_{2,j} \ar[r] & C_{2,j+1}
}
\]
commutes in $C^+(\BMod_A)$. Moreover, one can choose homotopies witnessing that each row is a complex in $K^+(\BMod_A)$ which themselves form commutative diagrams with the $f_j$.
\item[(b)]
For each $j$, there exists a morphism $g_j: D_{1,j} \to D_{2,j}$ in $C^+(\BMod_A)$
consisting of completely continuous morphisms in $\BMod_A$ such that the diagram
\[
\xymatrix{
D_{1,j} \ar^{g_j}[d] \ar^{\alpha_{1,j}}[r] & C_{1,j} \ar^{f_j}[d] \ar^{\beta_{1,j}}[r] & D_{1,j} \ar^{g_j}[d] \ar^{\alpha_{1,j}}[r] & C_{1,j} \ar^{f_j}[d] \\
D_{2,j} \ar^{\alpha_{1,j}}[r] & C_{2,j} \ar^{\beta_{2,j}}[r] & D_{2,j}\ar^{\alpha_{1,j}}[r] & C_{2,j}
}
\]
commutes in $C^+(\BMod_A)$. Moreover, one can choose homotopies witnessing that each row is a complex in $K^+(\BMod_A)$ which themselves form commutative diagrams with the $f_j$ and the $g_j$.
\item[(c)]
The morphism $\Tot(f_\bullet): \Tot(C_{1,\bullet}) \to \Tot(C_{2,\bullet})$
is a quasi-isomorphism.
\end{enumerate}
Then each cohomology group $h^i(\Tot(C_{1,\bullet}) \cong h^i(\Tot(C_{2,\bullet}))$ is contained in some finite $A$-module.
\end{lemma}
\begin{proof}
For $0\leq j\leq n$, we will construct by descending induction a quasi-isomorphism
\[
T_{i,j}:\Tot(C_{i,j}\to\cdots \to C_{i,n})\to \mathrm{T}(D)_{i,j}
\]
in $C^+(\BMod_A)$ as follows. For $j=n$, put $\mathrm{T}(D)_{i,n}=D_{i,n}$ and $T_{i,n}=\beta_{i,n}$. Suppose we have constructed $T_{i,j}$. Consider the diagram
\[
\xymatrix{
C_{i,j-1} \ar[r] \ar^{\beta_{i,j-1}}[d] & \Tot(C_{i,j}\to\cdots \to C_{i,n})\ar[d]^{T_{i,j}} \\
D_{i,j-1} \ar[r] & \mathrm{T}(D)_{i,j}
}
\] 
where $D_{i,j-1}\to\mathrm{T}(D)_{i,j}$ is given by the composition 
\[
D_{i,j-1} \stackrel{\alpha_{i,j-1}}\longrightarrow C_{i,j-1}\stackrel{T_{i,j}}\longrightarrow\Tot(C_{i,j}\to\cdots \to C_{i,n})\to\mathrm{T}(D)_{i,j}.
\]
Since $\alpha_{i,j-1}$ and $\beta_{i,j-1}$ are inverse to each other in $K^+(\BMod_A)$, this diagram is commutative in $K^+(\BMod_A)$. We set 
\[
\mathrm{T}(D)_{i,j-1}=\Cone(D_{i,j-1}\to\mathrm{T}(D)_{i,j})
\]
and $T_{i,j-1}$ to be the morphism 
\[
\Tot(C_{i,j-1}\to\cdots \to C_{i,n})=\Cone(C_{i,j-1} \to \Tot(C_{i,j}\to\cdots \to C_{i,n}))\to\mathrm{T}(D)_{i,j-1}.
\]
Since $\beta_{i,j-1}$ and $T_{i,j}$ are quasi-isomorphisms, $T_{i,j-1}$ is a quasi-isomorphism as well. Note that by
(b), the family $g_\bullet$ induces a morphism $\mathrm{T}(D)_{1,j}\to \mathrm{T}(D)_{2,j}$ making the diagram
\[
\xymatrix{
\mathrm{T}(D)_{1,j}\ar[r]\ar[d] & \Tot(C_{1,j}\to\cdots \to C_{1,n})\ar[d] \\
\mathrm{T}(D)_{2,j}\ar[r] & \Tot(C_{2,j}\to\cdots \to C_{2,n})
}
\]
commutative in $C^+(\BMod_A)$. By (c), we thus get a quasi-isomorphism $\mathrm{T}(D)_{1,0}\to \mathrm{T}(D)_{2,0}$, to which we may directly apply Lemma~\ref{L:Cartan-Serre}.
\end{proof}

\begin{remark} \label{R:Cartan-Serre use case2}
Lemma~\ref{L:Cartan-Serre homotopy} is required because the Proposition~\ref{P:direct image finiteness} involves a more complicated arrangement than the one described in Remark~\ref{R:Cartan-Serre use case}: we cannot represent the desired comparison of cohomology using a single completely continuous morphism in $C^+(\BMod_A)$.
Rather, what arises most naturally is a sequence of (pairs of) complexes, each derived as
Remark~\ref{R:Cartan-Serre use case} for a different ring homomorphism, which fit together in $K^+(\BMod_A)$ but not in $C^+(\BMod_A)$. Moreover, the comparison maps between the individual complexes are not quasi-isomorphisms; this only occurs for the total complexes.
The commutativity conditions will arise from the construction of the comparison maps as certain base extensions.
\end{remark}

\section{Cohomology of \texorpdfstring{$(\varphi, \Gamma)$}{(phi, Gamma)}-modules}
\label{subsec:ordinary finiteness}

As an example of the preceding discussion,  we rederive some existing finiteness results for the cohomology of $(\varphi, \Gamma)$-modules in classical (nonrelative) $p$-adic Hodge theory.

\begin{defn}
Let $\bA_{\QQ_p}$ be the ring of formal sums $\sum_{n \in \ZZ} c_n \pi^n$ with $c_n \in \ZZ_p$ such that $c_n \to 0$ as $n \to -\infty$. For $\gamma \in \ZZ_p^\times$, this ring admits commuting endomorphisms $\varphi, \gamma$ given by the formulas
\begin{equation} \label{eq:phi gamma}
\varphi\left(\sum_n c_n \pi^n\right) = \sum_n c_n ((1+\pi)^p-1)^n, \qquad
\gamma\left(\sum_n c_n \pi^n\right) = \sum_n c_n ((1+\pi)^\gamma-1)^n.
\end{equation}
A \emph{$(\varphi, \Gamma)$-module} over $\bA_{\QQ_p}$ is a finite projective module
over $\bA_{\QQ_p}$ equipped with commuting semilinear actions of $\varphi$ and $\Gamma$;
if we allow the module to be finitely generated but not necessarily projective, we obtain a  \emph{generalized $(\varphi, \Gamma)$-module} over $\bA_{\QQ_p}$.
We define the cohomology groups $H^i_{\varphi, \Gamma}(M)$ of a generalized $(\varphi, \Gamma)$-module $M$ to be the 
total cohomology of the double complex 
\[
C_{\varphi, \Gamma}(M): 0 \to C^{\bullet}_{\cont}(\Gamma, M) \stackrel{\varphi-1}{\to} 
C^{\bullet}_{\cont}(\Gamma, M) \to 0,
\]
where $C^{\bullet}_{\cont}(\Gamma, M)$ is the complex of continuous group cochains.
\end{defn}

\begin{lemma} \label{L:cohomology mod p}
Let $R$ be a $p$-adically separated and complete ring. Let $\calA$ be an abelian subcategory of the category of $p$-adically separated and complete $R$-modules such that
for any $M\in\calA$ which is finitely generated as an $R$-module, $\Tor_1^R(M, R/pR)$ is a finitely generated $R/pR$-module. Now suppose that
$C: 0 \to C^0 \to \cdots \to C^n \to 0$
is a bounded complex in $\calA$ such that the $R/pR$-modules $H^i(C^\bullet \otimes_R R/pR)$
are finitely generated.
Then the $R$-modules $H^i(C^\bullet)$ are themselves finitely generated.
\end{lemma}
\begin{proof}
We proceed by induction on $n$. 
Since base extension of modules is right exact, we have $H^n(C^\bullet \otimes_{R} R/p^m R) = H^n(C^\bullet) \otimes_{R} R/p^m R$; in particular, 
for $m_1 \leq m_2$, we have
\[
H^n(C^\bullet \otimes_{R} R/p^{m_2} R) \otimes_{R/p^{m_2}R} R/p^{m_1}R \cong
H^n(C^\bullet \otimes_{R} R/p^{m_1} R).
\]
By hypothesis, $H^n(C^\bullet \otimes_{R} R/p R)$ is a finitely generated $R/pR$-module; 
we may thus choose a finitely generated $R$-submodule $S$ of $C^n$ which surjects onto $H^n(C^\bullet \otimes_{R} R/p R)$. Now
\[
C^{n-1} \oplus S \to C^n
\]
is a map between $p$-adically separated and complete $R$-modules which is surjective modulo $p$, hence it is surjective; that is, $S \to H^n(C^\bullet)$ is surjective, so $H^n(C^\bullet)$ is finitely generated.

Now put $C^{\prime n-1} = \ker(C^{n-1} \to C^n)$; the truncated complex 
\[
C': 0 \to C^0 \to \cdots \to C^{n-2} \to C^{\prime n-1} \to 0
\]
is again a complex in $\calA$.
The $R/pR$-modules $H^i(C^{\prime \bullet} \otimes_R R/pR)$ are again finitely generated: namely, this
is obvious except for the final group
\[
H^{n-1}(C^{\prime \bullet} \otimes_R R/pR)
= 
\coker(C^{n-2}/pC^{n-2} \to C^{\prime n-1}/pC^{\prime n-1}),
\]
which fits into a right exact sequence
\[
\Tor_1^R(H^n(C^{\bullet}), R/pR)\to H^{n-1}(C^{\prime \bullet} \otimes_R R/pR)\to H^{n-1}(C^{\bullet} \otimes_R R/pR)\to0.
\]
By assumption we have that $\Tor_1^R(H^n(C^{\bullet}), R/pR)$ is a finitely generated $R/pR$-module. This yields that $H^{n-1}(C^{\prime \bullet} \otimes_R R/pR)$ is a finitely generated $R/pR$-module as well. 
We may thus apply the induction hypothesis to $C'$ to complete the proof.
\end{proof}

The following result
was proved for $p > 2$ by Herr \cite[Th\'eor\`eme~B]{herr1}
and in full by the second author \cite[Theorem~3.3]{liu-herr};
here, we derive it as an application of the Cartan-Serre method. 
\begin{theorem}[Herr, Liu] \label{T:integral ordinary phi-Gamma finiteness}
Let $M$ be a generalized $(\varphi, \Gamma)$-module over $\bA_{\QQ_p}$.
Then the groups $H^i_{\varphi,\Gamma}(M)$ are finite $\Zp$-modules for $i=0,1,2$ and zero for $i>2$.
\end{theorem}
\begin{proof}
Write $\Gamma$ as the product of a finite subgroup $\Delta$ and a subgroup $\Gamma'$ admitting a topological generator $\gamma$. Then
$C^{\bullet}_{\cont}(\Gamma, M)$ is quasi-isomorphic to the complex
\[
0 \to M^\Delta \stackrel{\gamma-1}{\to} M^\Delta \to 0
\]
so its cohomology vanishes outside of degrees 0 and 1; this implies that
$H^i_{\varphi,\Gamma}(M) = 0$ for $i > 2$.

To prove that the groups $H^i_{\varphi,\Gamma}(M)$ are finite $\Zp$-modules for $i=0,1,2$, by Lemma~\ref{L:cohomology mod p} we may reduce to the case where $M$ is killed by $p$,
and hence finite free over $\bA_{\QQ_p}/(p) \cong \QQ_p((\pi))$.
Fix a basis of $M$. For $0 < s \leq r < 1$, note that $\QQ_p((\pi))$ is complete with respect to the norm $\left| \bullet \right|_{s,r}$ given by
\[
\left| c_n \pi^n \right|_{s,r} = \max\{\max\{r^n, s^n\}: n \in \ZZ, c_n \neq 0\},
\]
so $M$ is complete with respect to the supremum norm $\left| \bullet \right|_{s,r,M}$.
For any $s,r$, the morphisms in $C^{\bullet}_{\cont}(\Gamma, M)$ are bounded with respect to $\left| \bullet \right|_{s,r,M}$; if $s \leq r^p$, the the morphism $\varphi-1$ is bounded with respect to $\left| \bullet \right|_{s,r^p,M}$ on the source and
$\left| \bullet \right|_{s,r,M}$ on the target.
Let $C_{\varphi, \Gamma}(M)_{s,r}$ be the double complex $C_{\varphi, \Gamma}(M)$ equipped with norms in this manner.

Now choose $r,s,r',s' \in (0,1)$ with $s \leq r^p$, $s' < s$, and $r < r'$.
Then
\[
C_{\phi, \Gamma}(M)_{s',r'} \to C_{\phi, \Gamma}(M)_{s,r}
\]
is a morphism of double complexes which is a quasi-isomorphism (since it is the identity map on underlying modules) in which each individual map is completely continuous.
We would then be in the situation where the desired finiteness would follow from
Lemma~\ref{L:Cartan-Serre}, except that the base ring $\FF_p$ does not contain a topologically nilpotent unit. However, we may get around this issue using the method of
\cite[Lemma~8.9]{kedlaya-reified}: extend base from $\FF_p$ to $\FF_p((z))$, apply
Lemma~\ref{L:Cartan-Serre} (keeping in mind Remark~\ref{R:Cartan-Serre use case}), then project onto the coefficient of $z^0$. We thus recover the desired finiteness.
\end{proof}

\begin{defn}
Let $\bC_{\QQ_p}$ be the union of the rings of rigid analytic functions on the discs
$c < \left| \pi \right| < 1$ over all $c \in (0,1)$ (commonly known as the \emph{Robba ring} over $\QQ_p$); this ring carries a natural locally convex topology as the inductive limit of Fr\'echet spaces. The equations \eqref{eq:phi gamma} again define commuting operators $\varphi, \gamma$ on $\bC_{\QQ_p}$ for $\gamma \in \Gamma$.
(The notation is in the style of \cite{part1, part2}; another common label is $\bB^{\dagger}_{\rig,\QQ_p}$.) 

A \emph{$(\varphi, \Gamma)$-module} over $\bC_{\QQ_p}$ is a finite projective module
over $\bC_{\QQ_p}$ equipped with commuting semilinear actions of $\varphi$ and $\Gamma$,
with the action of $\Gamma$ being continuous;
if we allow the module to be finitely presented but not necessarily projective, we obtain a \emph{generalized $(\varphi, \Gamma)$-module} over $\bC_{\QQ_p}$.
We define cohomology groups as before.
\end{defn}

The following statement is due to the second author \cite{liu-herr}, but again we prefer to prove it as an illustration of the Cartan-Serre method.

\begin{theorem}[Liu] \label{T:ordinary phi-Gamma finiteness}
Let $M$ be a generalized $(\varphi, \Gamma)$-module over $\bC_{\QQ_p}$.
Then the groups $H^i_{\varphi,\Gamma}(M)$ are finite $\Qp$-modules for $i=0,1,2$ and zero for $i>2$.
\end{theorem}
\begin{proof}
As in the proof of Theorem~\ref{T:integral ordinary phi-Gamma finiteness}, we see that
$H^i_{\varphi,\Gamma}(M) = 0$ for $i > 2$. To prove that the groups $H^i_{\varphi,\Gamma}(M)$ are finite $\Qp$-modules for $i=0,1,2$, 
for any interval $I$,
let $A(I)$ denote the rigid analytic annulus $\left| \pi \right| \in I$ over $\QQ_p$,
omitting the parentheses when $I$ is written as an explicit interval with endpoints.
For some $c_1 > 0$, we may 
then realize $M$ as a module $M_{[c_1,1)}$ over $\calO(A[c_1,1))$ equipped with an action of $\Gamma$ and a compatible isomorphism 
\[
\varphi^* M_{[c_1,1)} \cong M_{[c_1,1)} \otimes_{\calO(A[c_1,1))} \calO(A[c_1^{1/p},1)).
\]
For $d_1 \in [c_1^{1/p},1)$, put 
\[
M_{[c_1,d_1]} = M_{[c_1,1)} \otimes_{\calO(A[c_1,1))} \calO(A[c_1,d_1]);
\]
by \cite[Theorem~5.7.11]{part2}, the cohomology of 
$C_{\varphi, \Gamma}(M)$
is also computed by the double complex
\[
0 \to C^\bullet_{\cont}(\Gamma, M_{[c_1,d_1]}) \stackrel{\varphi-1}{\to} C^\bullet_{\cont}(\Gamma, M_{[c_1^{1/p},d_1]}) \to 0.
\]
We may thus deduce the claim by computing the cohomology using two intervals $[c_1, d_1], [c_2, d_2]$ with one contained in the interior of the other, then applying
Lemma~\ref{L:Cartan-Serre} (keeping in mind Remark~\ref{R:Cartan-Serre use case}).
\end{proof}

\begin{remark}
Both Theorem~\ref{T:integral ordinary phi-Gamma finiteness}
and Theorem~\ref{T:ordinary phi-Gamma finiteness} promote immediately to
$(\varphi^a, \Gamma)$-modules for any positive integer $a$: if $M$ is such an object, then
$M \oplus \varphi^* M \oplus \cdots \oplus (\varphi^{a-1})^* M$ carries an action of $\varphi$ and its $(\varphi, \Gamma)$-cohomology coincides with the $(\varphi^a, \Gamma)$-cohomology of $M$. Similar considerations will hold later in the paper; to simplify the exposition, we make all statements exclusively for $(\varphi, \Gamma)$-modules.
\end{remark}

\begin{remark}
Theorem~\ref{T:integral ordinary phi-Gamma finiteness}
and Theorem~\ref{T:ordinary phi-Gamma finiteness} also promote immediately to theorems
concerning $(\varphi, \Gamma)$-modules for which the base field $\QQ_p$ has been replaced by a finite extension (this being the level of generality considered in \cite{herr1},
\cite{liu-herr}, and elsewhere).
\end{remark}

\begin{remark}
The argument of Theorem~\ref{T:ordinary phi-Gamma finiteness} can also be used to obtain the finiteness theorem for cohomology of arithmetic families of $(\varphi, \Gamma)$-modules given by the first author with Pottharst and Xiao in \cite{kpx}. The proof therein is more subtle; it involves reducing to the result of \cite{liu-herr} via an intricate d\'evissage argument and a duality computation.
\end{remark}

\section{Relative \texorpdfstring{$(\varphi, \Gamma)$}{(phi, Gamma)}-modules}

We continue by summarizing the main points of the theory of relative $(\varphi, \Gamma)$-modules on rigid analytic spaces, as developed in \cite{part1, part2}.

\begin{defn}
For $X$ an adic space, let $X_{\proet}$ denote the \emph{pro-\'etale site} in the sense of \cite[\S 3]{scholze2}, \cite[\S 9.1]{part1}; that is, a basic pro-\'etale open in $X$ is a tower of finite \'etale surjective covers over an \'etale open. (We do not consider the finer pro-\'etale topology used in \cite{scholze-berkeley}.)
If $p$ is topologically nilpotent on $X$, then the pro-\'etale topology is generated by
towers of finite \'etale covers of affinoids for which the completed direct limit of the associated rings is a perfectoid ring \cite[Lemma~3.3.26]{part2}; such pro-\'etale opens are called \emph{perfectoid subdomains} of $X$ (or better, of $X_{\proet}$).
\end{defn}

We recall some basic facts about pseudocoherent modules from \cite{part2}.
\begin{defn} \label{D:pseudocoherent}
A module $M$ over a ring $R$ is \emph{pseudocoherent} if it admits a projective resolution (possibly of infinite length) by finite projective modules. For example, any finitely generated module over a noetherian ring is pseudocoherent, as is any finitely presented module over a coherent ring.
A pseudocoherent module over a Banach ring is complete for its natural topology \cite[Corollary~1.2.9]{part2}.
Pseudocoherent modules form a stack over stably uniform adic Banach rings for the analytic topology and the \'etale topology \cite[Theorem~2.5.6, Corollary~2.5.7]{part2}
and over perfectoid adic Banach rings 
for the pro-\'etale topology \cite[Corollary~3.4.9]{part2}.

A module $M$ over $R$ is \emph{pseudoflat} if for every $R$-module $N$ admitting a partial projective resolution $P_2 \to P_1 \to P_0 \to M \to 0$ with $P_0, P_1, P_2$ finite projective, we have $\Tor_1^R(M,N) = 0$; this implies that tensoring with $M$ is an exact functor from pseudocoherent $R$-modules to $R$-modules (or to pseudocoherent $M$-modules if $M$ is itself an $R$-algebra).
For example, every flat module is pseudoflat, and conversely if $R$ is coherent (because omitting the condition of finiteness of $P_2$ leads back to the definition of a flat module). While rational localizations of affinoid algebras are flat, rational localizations of arbitrary adic Banach rings are only known to be pseudoflat
\cite[Theorem~2.4.8]{part2}.
\end{defn}

\begin{defn}
For any adic space $X$ on which $p$ is topologically nilpotent, define the $p$-adic period sheaves $\tilde{\bA}_X$ and $\tilde{\bC}_X$ on the pro-\'etale site $X_{\proet}$ as in \cite[Definition~9.3.3]{part1}; by construction, both admit a \emph{Frobenius endomorphism} $\varphi$. 
For real numbers $r,s$ such that $0 < s \leq r$, we also have a sheaf
$\tilde{\bC}^{[s,r]}_X$ on $X_{\proet}$ defined as in \cite[Definition~9.3.3]{part1}.
We have an inclusion $\tilde{\bC}^{[s,r]}_X \to \tilde{\bC}^{[s',r']}_X$ whenever
$[s',r'] \subseteq [s,r]$, and a Frobenius endomorphism
$\varphi: \tilde{\bC}^{[s,r]}_X \to \tilde{\bC}^{[s/p,r/p]}_X$.
Pseudocoherent $\tilde{\bA}_X$-modules and pseudocoherent $\tilde{\bC}^{[s,r]}_X$-modules form stacks over perfectoid adic Banach rings for the pro-\'etale topology
\cite[Theorem~4.3.3]{part2}.
\end{defn}

\begin{remark}
As a quick summary, we recall that for $Y$ a perfectoid subdomain of $X$, $\tilde{\bA}_X$ evaluates to the ring of Witt vectors over the ring of sections of the tilted completed structure sheaf on $Y$,
while $\tilde{\bC}^{[s,r]}_X$ evaluates to an ``extended Robba ring'' derived from the Witt vectors by restricting to elements satisfying a growth condition, inverting $p$, then completing for a suitable Banach norm. The evaluation of $\tilde{\bC}_X$ is obtained from the evaluations of the $\tilde{\bC}^{[s,r]}_X$ by first taking an inverse limit as $s \to 0^+$, then a direct limit as $r \to 0^+$.
\end{remark}

\begin{remark} \label{R:pseudoflat}
For $Y$ a perfectoid subdomain of a perfectoid space $X$,
the morphism $\tilde{\bA}_{\tilde{X}} \to \tilde{\bA}_{\tilde{X}'}$
(resp. $\tilde{\bC}^{[s,r]}_{\tilde{X}} \to \tilde{\bC}^{[s,r]}_{\tilde{X}'}$)
is pseudoflat \cite[Proposition~4.2.3]{part2}.
\end{remark}

\begin{lemma} \label{L:base extension cohomology Robba}
Let $(A,A^+)$ be a perfectoid adic Banach ring. Let $\Spa(B,B^+)$ be a perfectoid subdomain of $\Spa(A,A^+)$. Let $C^\bullet$ be a complex of Banach modules over $\tilde{\bC}^{[s,r]}_A$.
Suppose that $C$ is a perfect complex of $\tilde{\bC}^{[s,r]}_A$-modules with pseudocoherent cohomology
and that $C \widehat{\otimes}_{\tilde{\bC}^{[s,r]}_A} \tilde{\bC}^{[s,r]}_B$ is a perfect complex of $\tilde{\bC}^{[s,r]}_B$-modules with pseudocoherent cohomology. Then the natural morphisms 
\[
h^\bullet(C) \otimes_{\tilde{\bC}^{[s,r]}_A} \tilde{\bC}^{[s,r]}_B \to h^\bullet(C \widehat{\otimes}_{\tilde{\bC}^{[s,r]}_A} \tilde{\bC}^{[s,r]}_B)
\]
are isomorphisms.
\end{lemma}
\begin{proof}
Upon taking the completed tensor product
over $\QQ_p$ with the $p$-cyclotomic extension $L$, the rings
$\tilde{\bC}^{[s,r]}_A$, $\tilde{\bC}^{[s,r]}_B$  become perfectoid \cite[Proposition~4.1.13]{part2} and the map between them becomes a perfectoid subdomain 
\cite[Lemma~4.2.2]{part2}. We may then apply Remark~\ref{R:pseudoflat}, then use a splitting $L \to \QQ_p$ of the inclusion to recover the desired isomorphisms.
\end{proof}

\begin{defn}
A \emph{$(\varphi, \Gamma)$-module} (resp.\ a \emph{pseudocoherent $(\varphi, \Gamma)$-module})
over $\tilde{\bA}_X$ or $\tilde{\bC}_X$ on $X$ is a sheaf $\calF$ of modules which is locally represented by a finite projective (resp.\ pseudocoherent) module over the corresponding ring, together with an isomorphism $\varphi^* \calF \to \calF$. We sometimes refer to a $(\varphi, \Gamma)$-module also as a \emph{projective $(\varphi, \Gamma)$-module} for emphasis.
Denote by $\CPhi_X$ the category of pseudocoherent $(\varphi, \Gamma)$-modules over $\tilde{\bC}_X$.
\end{defn}

\begin{remark}
In this context, the symbol $\Gamma$ in the term \emph{$(\varphi, \Gamma)$-module} is but a skeuomorph: the former role of $\Gamma$, as a provider of descent data, is now assumed by the pro-\'etale topology.
\end{remark}

\begin{remark} \label{R:coherent type A}
The module structure of pseudocoherent $(\varphi, \Gamma)$-modules over $\tilde{\bA}_X$ is quite simple: the only possible torsion is $p$-power torsion, and a $p$-torsion-free object is projective \cite[Lemma~4.5.4]{part2} (see Theorem~\ref{T:finitely generated to pseudocoherent type A} for an even stronger statement).
In particular, without risk of confusion, we may refer to these objects also as \emph{coherent} $(\varphi, \Gamma)$-modules over $\tilde{\bA}_X$. The situation in type $\tilde{\bC}$ is a bit subtler; see Remark~\ref{R:coherent type C}.
\end{remark}

\begin{remark} \label{R:type C interval}
In case $s \leq r/p$, the category $\CPhi_X$ is equivalent to the category of
pseudocoherent $\tilde{\bC}^{[s,r]}_X$-modules $\calF_{[s,r]}$ equipped with isomorphisms
\[
\varphi^* \calF_{[s,r]} \otimes_{\tilde{\bC}^{[s/p,r/p]}_X} \tilde{\bC}^{[s,r/p]}_X
\cong
\calF_{[s,r]} \otimes_{\tilde{\bC}^{[s,r]}_X} \tilde{\bC}^{[s,r/p]}_X;
\]
see \cite[Theorem~4.6.10]{part2}. 
(We sometimes refer to these objects as \emph{pseudocoherent $(\varphi, \Gamma)$-modules}
over $\tilde{\bC}^{[s,r]}_X$.)
In particular, for any perfectoid subdomain $Y$ of $X$,
any pseudocoherent $(\varphi, \Gamma)$-module over $X$ is represented by a pseudocoherent $\tilde{\bC}_X(Y)$-module.
Moreover, if $\calF \in \CPhi_X$ corresponds to $\calF_{[s,r]}$ in this fashion, then the complexes
\[
0 \to \calF \stackrel{\varphi-1}{\to} \calF \to 0,
\qquad
0 \to \calF_{[s,r]} \stackrel{\varphi-1}{\to} \calF_{[s,r/p]} \to 0
\]
have the same cohomology sheaves \cite[Theorem~4.6.9]{part2}.
\end{remark}

\begin{defn}
By a \emph{$\ZZ_p$-local system} or \emph{$\QQ_p$-local system} over an adic space $X$,
we will mean a finite free module over the corresponding locally constant sheaf
$\underline{\ZZ_p}$ or $\underline{\QQ_p}$ (defining locally constant sheaves in the sense of \cite[Definition~1.4.10]{part1}).
\end{defn}

\begin{theorem} \label{T:Artin-Schreier}
The functor $T \mapsto \calF = T \otimes_{\underline{\ZZ_p}} \tilde{\bA}_X$ defines an equivalence of categories between locally finite $\underline{\ZZ_p}$-modules on $X_{\proet}$
and pseudocoherent $(\varphi, \Gamma)$-modules
over $\tilde{\bA}_X$, with the one-sided inverse being $\calF \mapsto \calF^{\varphi}$; in particular, $\ZZ_p$-local systems correspond to projective $(\varphi, \Gamma)$-modules.
Moreover, the pro-\'etale cohomology of $T$ is computed by the $(\varphi, \Gamma)$-hypercohomology of $\calF$.
\end{theorem}
\begin{proof}
See \cite[Corollary~4.5.8]{part2}.
\end{proof}

\begin{theorem} \label{T:Artin-Schreier rational}
The functor $V \mapsto \calF = V \otimes_{\underline{\QQ_p}} \tilde{\bC}_X$ defines a fully faithful functor from $\QQ_p$-local systems on $X$ to $(\varphi, \Gamma)$-modules
over $\tilde{\bC}_X$, with the one-sided inverse being $\calF \mapsto \calF^{\varphi}$.
Moreover, the pro-\'etale cohomology of $V$ is computed by the $(\varphi, \Gamma)$-hypercohomology of $\calF$.
\end{theorem}
\begin{proof}
See \cite[Theorem~4.5.11]{part2}.
\end{proof}

\begin{theorem} \label{T:finitely generated to pseudocoherent type A}
Let $X$ be an affinoid space over $K$.
Let $\psi$ be a finite \'etale tower over $X$ whose total space $Y$ is perfectoid. 
Let $\calF$ be a sheaf of $\tilde{\bA}_X$-modules. Then $\calF$ is pseudocoherent if and only if its
 restriction to $Y$ is represented by a finitely generated $\tilde{\bA}(Y)$-module.
\end{theorem}
\begin{proof}
See \cite[Proposition~8.3.5]{part2}. 
\end{proof}

\begin{theorem} \label{T:finitely generated to pseudocoherent}
Let $X$ be an affinoid space over $K$.
Let $\psi$ be a finite \'etale tower over $X$ whose total space $Y$ is perfectoid. 
Let $\calF$ be a sheaf of $\tilde{\bC}^{[s,r]}_X$-modules.
Then $\calF$ is pseudocoherent if and only if its restriction to $Y$ is represented by a finitely generated $\tilde{\bC}^{[s,r]}(Y)$-module.
\end{theorem}
\begin{proof}
See \cite[Proposition~8.9.2]{part2}. 
\end{proof}

\begin{theorem} \label{T:pseudocoherent noetherian}
Let $X$ be an affinoid space over $K$.
\begin{enumerate}
\item[(a)]
The category $\CPhi_X$ of pseudocoherent $(\varphi, \Gamma)$-modules over $\tilde{\bC}_X$ is an abelian subcategory of the category of sheaves of $\tilde{\bC}_X$-modules on $X_{\proet}$ (that is, the formation of kernels and cokernels commutes with the embedding).
\item[(b)]
If $X$ is quasicompact, then $\CPhi_X$ satisfies the ascending chain condition: given any sequence
$\calF_0 \to \calF_1 \to \cdots$ of epimorphisms in $\CPhi_X$, there exists $i_0 \geq 0$ such that for all $i \geq i_0$, the map $\calF_i \to \calF_{i+1}$ is an isomorphism.
\item[(c)]
For $i \geq 0$, the bifunctors $\Ext^i$ and $\Tor_i$ take $\CPhi_X \times \CPhi_X$ into $\CPhi_X$.
\item[(d)]
For any morphism $f: Y \to X$ of rigid spaces over $K$, for $i \geq 0$, the functor $L_i f_{\proet}^*$ takes $\CPhi_X$ into $\CPhi_Y$.
\end{enumerate}
\end{theorem}
\begin{proof}
See \cite[Theorem~8.10.6, Theorem~8.10.7]{part2}.
\end{proof}

\begin{remark} \label{R:coherent type C}
In light of Theorem~\ref{T:pseudocoherent noetherian}, when $X$ is a rigid space over a field rather than a more general adic space, it is natural to refer to objects of $\CPhi_X$ also as \emph{coherent $(\varphi, \Gamma)$-modules} over $\tilde{\bC}_X$.
The proof of Theorem~\ref{T:pseudocoherent noetherian} yields some additional information about the module structure of these objects. This is more easily described in terms
of the corresponding modules over $\tilde{\bC}^{[s,r]}_X$: any such module admitting a $\Gamma$-action (i.e., which is a sheaf for the pro-\'etale topology)
becomes finite projective after inverting some element $t$ of $\tilde{\bC}^{[s,r]}_L$ for $L$ a perfectoid extension of $K$.

In case $1 \in [s,r]$, one candidate for such an element is an element $t_\theta$ generating the kernel of the canonical map
$\theta: \tilde{\bC}^{[s,r]}_L \to L$. In case $K$ is discretely valued, 
the element $t$ can always be taken to be a product of images of $t_\theta$ under various powers of $\varphi$; by contrast, when $K$ is perfectoid, we have many more options for $t$. However, most of the module-theoretic complexity arises from $t_\theta$, in the following sense: if $t$ is irreducible and coprime to $t_\theta$, then any $t$-torsion coherent $\Gamma$-module over $\tilde{\bC}^{[s,r]}_X$ is finite projective
over $\tilde{\bC}^{[s,r]}_X/(t)$ \cite[Corollary~8.8.10]{part2}.

By contrast, if $t = t_\theta$, then $t$-torsion coherent $\Gamma$-modules over $\tilde{\bC}^{[s,r]}_X$ need not be finite projective over $\tilde{\bC}^{[s,r]}_X/(t) \cong \widehat{\calO}_X$; for instance, the pullback of any coherent $\calO_X$-module occurs in this category. However, one can at least say that these modules have Fitting ideals which descend to $\calO_X$ (see \cite[\S 8.3]{part2}), so one can establish the desired module-theoretic properties using noetherian induction and resolution of singularities.
\end{remark}

\begin{defn} \label{D:co-theta-projective}
We say that a pseudocoherent $(\varphi, \Gamma)$-module $\calF$ over $\tilde{\bC}^{[s,r]}_X$ is \emph{$\theta$-local} (resp.\ \emph{co-$\theta$-projective})
if it is annihilated by (resp.\ becomes projective after inverting) some product
of images of $t_\theta$ under various powers of $\varphi$, and similarly for an object of $\CPhi_X$.

By Remark~\ref{R:coherent type C}, if $K$ is discretely valued, then every object of $\CPhi_X$ is co-$\theta$-projective; moreover, for general $K$, most module-theoretic difficulties are concentrated in the subcategory of $\theta$-local objects. For example,
by \cite[Corollary~8.8.10]{part2} (as applied in Remark~\ref{R:coherent type C}), in the quotient of $\CPhi_X$ by the Serre subcategory of  $\theta$-local objects, every object has projective dimension at most 1; in other words,
any torsion-free object of $\CPhi_X$ is co-$\theta$-projective.
By similar arguments, one sees that the essential images of $\Ext^i$ for $i>1$,  $\Tor_i$ for $i>1$, and $L_i f^*_{\proet}$ for $i>0$ all consist of $\theta$-local objects
\cite[Theorem~8.10.6, Theorem~8.10.7]{part2}.
\end{defn}

\begin{remark}
Although we will not use this result here, we note that the ascending chain condition also holds for the localization of $\CPhi_X$ at a point of $X$. See 
\cite[Theorem~8.10.9]{part2}.
\end{remark}

The following example shows that one cannot hope to extend
Theorem~\ref{T:pseudocoherent noetherian} to arbitrary higher direct images along a general morphism of rigid spaces.
\begin{example} \label{exa:bad closed unit disc}
Put $X = \Spa(K,K^\circ)$, $Y = \Spa(K\{T,T^{-1}\}, K\{T,T^{-1}\}^{\circ})$, and
$\calF = \tilde{\bC}_Y$. Then the higher direct images of $\calF$ along $Y_{\proet} \to X_{\proet}$ are not coherent $(\varphi, \Gamma)$-modules over $\tilde{\bC}_X$. This amounts to the fact that for $R$ a Banach algebra over $\Qp$ containing a topologically nilpotent unit $\pi$, if we equip $S = R\{T,T^{-1}\}$ with the action of $\Zp^\times$ sending $T$ to $(1+\pi)T$, then for any $\gamma \in \Zp^\times$ of infinite order, the action of $\gamma-1$ on $S$ is not strict. A closely related fact is that the de Rham cohomology of the closed unit disc over a $p$-adic field is not finite-dimensional: integration of a power series preserves the radius of convergence but not the behavior at the boundary.
\end{example}

A related example is the following.

\begin{example} \label{exa:bad open annulus}
Let $Y$ be an open annulus over $X$.
Then the failure of finiteness described in Example~\ref{exa:bad closed unit disc} does not arise, but there are additional subtleties. For instance, for each character of $\Zp^\times$, we may define a $\Gamma$-module over $\tilde{\bC}_Y$ which is free of rank 1 with the action of
$\Zp^\times$ on the generator being given by the chosen character. When this character is exponentiation by a $p$-adic Liouville number, we expect the $\Gamma$-cohomology not to be finite-dimensional, again by analogy with a corresponding pathology in the theory of $p$-adic differential equations (see \cite[Chapter~13]{kedlaya-course} for further discussion).

However, note that this example does not admit an action of $\varphi$. One may thus still hope to prove coherence of higher direct images of coherent $(\varphi, \Gamma)$-modules along $Y_{\proet} \to X_{\proet}$. This might be likened to the fact that finiteness results for the rigid cohomology of overconvergent isocrystals on varieties of characteristic $p$ are generally only known in the presence of Frobenius structures
(e.g., see \cite{kedlaya-rigid-finiteness}). While this makes such a result plausible, we will only prove such a theorem for smooth proper morphisms.
\end{example}

\section{Properness for rigid analytic varieties}
\label{subsec:properness}

We next recall Kiehl's definition of properness for rigid analytic varieties,
and its interaction with completely continuous morphisms.

\begin{defn} \label{D:inner morphism}
Let $f: (B,B^+) \to (B', B^{\prime +})$ be an affinoid localization of adic Banach algebras over an adic Banach ring $(A,A^+)$.
We say that $f$ is \emph{inner (relative to $(A,A^+)$)} if there exists a strict surjection $A\{S\} \to B$ for some (possibly infinite set) $S$ such that each element of $S$ maps to a topologically nilpotent element of $B'$. Since this condition does not depend on the plus subrings, we will also say that $B \to B'$ is inner relative to $A$.
\end{defn}

\begin{lemma} \label{L:transfer inner}
Let $(A,A^+) \to (A',A^{\prime +}) \to (A'', A^{\prime \prime +})$ be morphisms of perfectoid adic Bananch algebras, and apply the perfectoid correspondence 
\cite[Theorem~3.3.8]{part2} to obtain morphisms $(R, R^+) \to (R', R^{\prime +}) \to (R'', R^{\prime \prime +})$ of perfect uniform adic Banach algebras.
If $(R', R^{\prime +}) \to (R'', R^{\prime \prime +})$ is an inner rational localization relative to $(R, R^+)$,
then $(A', A^{\prime +}) \to (A'', A^{\prime \prime +})$ is an inner rational localization relative to $(A, A^+)$.
\end{lemma}
\begin{proof}
By \cite[Theorem~3.3.18]{part2}, $(A', A^{\prime +}) \to (A'', A^{\prime \prime +})$ is also a rational localization.
Choose a strict surjection $f: R\{S\} \to R'$ such that each element of $S$ maps to a topologically nilpotent element of $R''$.
We then obtain a strict surjection $A\{S\} \to A'$ sending $s$ to $\theta([f(s)])$;
hence $(A', A^{\prime +}) \to (A'', A^{\prime \prime +})$ is inner.
\end{proof}

\begin{remark}
We do not know if the converse to Lemma~\ref{L:transfer inner} holds. An argument to prove this would likely encounter difficulties similar to those described in \cite[Remark~3.6.12]{part1}.
\end{remark}

\begin{hypothesis} \label{H:properness}
For the remainder of \S\ref{subsec:properness}, let $(A,A^+)$ be an adic affinoid algebra over $K$ (a nonarchimedean field of mixed characteristics).
\end{hypothesis}

\begin{remark} \label{R:inner morphism}
Let $f: (B, B^+) \to (B', B^{\prime +})$ be an affinoid localization of adic affinoid algebras over $(A,A^+)$.
In light of Hypothesis~\ref{H:properness},
the condition that $f$ is inner can be reformulated in the following ways.
\begin{enumerate}
\item[(a)]
The Banach ring $B$ is topologically generated over $A$ by power-bounded elements which become topologically nilpotent in $B'$.
\item[(b)]
For some $n \geq 0$, there exists a strict $A$-linear surjection $A\{T_1,\dots,T_n\} \to B$ which extends to a morphism $A\{T_1/r,\dots,T_n/r\} \to B'$ for some $r \in (0,1)$.
\item[(c)]
For the spectral seminorms on $A,B,B'$, the image of $\kappa_B$ in $\kappa_{B'}$ is a finite $\kappa_A$-algebra.
\end{enumerate}
\end{remark}

\begin{example}
For $r_i, s_i \in p^\QQ$ with $s_i < r_i$, the morphism
\[
A\{T_1/r_i,\dots,T_n/r_n\} \to A\{T_1/s_1, \dots, T_n/s_n\}
\]
is inner relative to $A$.
\end{example}

\begin{lemma} \label{L:inner completely continuous}
Let $f: (B,B^+) \to (B', B^{\prime +})$ be an inner affinoid localization of adic affinoid algebras over $(A, A^+)$.
Then the underlying morphism $B \to B'$ of Banach spaces over $A$ is completely continuous.
\end{lemma}
\begin{proof}
See the proof of \cite[Satz~2.5]{kiehl-finiteness}. 
\end{proof}

\begin{lemma} \label{L:finite etale inner}
Let $f: (B,B^+) \to (B',B^{\prime +})$ be an inner affinoid localization of affinoid algebras over $(A,A^+)$.
Let $g: (C,C^+) \to (C',C^{\prime +})$ be the base extension of $f$ along an integral morphism $B \to C$ of affinoid algebras over $A$.
Then $g$ is also inner.
\end{lemma}
\begin{proof}
This is immediate from Remark~\ref{R:inner morphism}(c).
\end{proof}

\begin{defn} \label{D:proper}
Let $f: Y \to X$ be a morphism of rigid analytic spaces over $K$.
We say that $f$ is \emph{separated} if the diagonal morphism $f: Y \to Y \times_X Y$ is a closed immersion. We say that $f$ is \emph{proper} if $f$ is separated and additionally, for every affinoid subspace $U$ of $X$, there exist two finite coverings $\{V_i\}_{i=1}^n$, $\{W_i\}_{i=1}^n$ of $f^{-1}(U)$ by affinoid subspaces with the property that for $i=1,\dots,n$, $W_i$ is an affinoid subdomain of $V_i$ which is inner relative to $U$. (That is, for $A, B', B''$ the rings of global sections of $U,V_i,W_i$, the morphism $B' \to B''$ is an affinoid localization which is inner relative to $A$.)
\end{defn}

\begin{remark}
The definition of properness given in Definition~\ref{D:proper} is the original definition of Kiehl \cite{kiehl-finiteness}. Alternate characterizations are described in \cite{temkin-local, temkin-local2}; such characterizations can be used to check for instance that properness is stable under composition, and that the analytification of a proper algebraic variety is proper.
\end{remark}

\begin{remark} \label{R:annuli cover discs}
It will be useful in what follows to recall that for any adic Banach algebra $(A,A^+)$, the disc $\Spa(A\{T\}, A^+\{T\})$ is covered by the annuli
$\Spa(A\{T^{\pm}\}, A^+\{T^{\pm}\})$ and $\Spa(A^+\{(T-1)^{\pm}\}, A^+\{(T-1)^{\pm}\})$.
\end{remark}

\begin{lemma} \label{L:proper covering}
Let $(A,A^+)$ be an adic affinoid algebra over $K$, and let
$f: X \to \Spa(A,A^+)$ be a smooth proper morphism of relative dimension $n$.
Then for each $x \in X$ and each affinoid subdomain $x \in V_0$ of $X$, there exist
a rational localization $(A,A^+) \to (B,B^+)$ with $f(x) \in \Spa(B,B^+)$, some affinoid subdomains $x \in V_2 = \Spa(C_2,C_2^+) \subseteq V_1 = \Spa(C_1,C_1^+) \subseteq V_0$ of $X$, and a commutative diagram of the form
\begin{equation} \label{eq:proper covering}
\xymatrix{
B\{\frac{s_{1,1}}{T_1},\dots,\frac{s_{1,n}}{T_n},\frac{T_1}{r_{1,1}},\dots,\frac{T_n}{r_{1,n}}\} \ar[d]
\ar[r] & B\{\frac{s_{2,1}}{T_1},\dots,\frac{s_{2,n}}{T_n},\frac{T_1}{r_{2,1}},\dots,\frac{T_n}{r_{2,n}}\} \ar[d] \\
C_1 \ar[r] & C_2 \\
}
\end{equation}
for some $0 < s_{1,i} < s_{2,i} \leq r_{2,i} < r_{1,i}$ in which the vertical arrows are finite \'etale morphisms of $A$-algebras.
In particular, $(C_1,C_1^+) \to (C_2,C_2^+)$ is inner relative to $(B,B^+)$.
\end{lemma}
\begin{proof}
Since $X$ is proper, we can find affinoid subdomains $x \in V_2' \subseteq V_1'$ of $X$ such that $V_2'$ is inner in $V_1'$ relative to $A$.
Since $X$ is also smooth, we can find an affinoid subdomain $x \in V_2'' \subseteq V_2' \cap V_0$ such that $V_2 = \Spa(C''_2,C_2^{\prime \prime +})$
and $C_2''$ is finite \'etale over $B\{T_1,\dots,T_n\}$ for some rational localization $(A,A^+) \to (B,B^+)$. As in Remark~\ref{R:annuli cover discs}, we may ensure that $T_1,\dots,T_n$ do not vanish at $x$.

Let $f_1,\dots,f_n \in C_2''$ be the pullbacks of $T_1,\dots,T_n$.
Then any sufficiently close approximations of $f_1,\dots,f_n$ also define a finite \'etale morphism $B\{T_1,\dots,T_n\} \to C_2''$. In particular, these approximations may be chosen so that this morphism extends to a commutative diagram
\[
\xymatrix{
B\{\frac{T_1}{r_{1,1}},\dots,\frac{T_n}{r_{1,n}}\} \ar[d]
\ar[r] & B\{T_1,\dots,T_n\} \ar[d] \\
C_1'' \ar[r] & C_2'' 
}
\]
for some $r_{1,1},\dots,r_{1,n} > 1$ in such a way that 
$\Spa(C_2'', C_2^{\prime \prime +}) \subseteq \Spa(C_1'', C_1^{\prime \prime +})$
is an inner rational subdomain inclusion of affinoid subdomains of $X$
and $B\{T_1/r_{1,1},\dots,T_n/r_{1,n}\} \to C_1''$ is finite \'etale. We thus obtain a diagram of the desired form by taking $r_{2, i} = 1$ and choosing $s_{1,i}, s_{2,i}$ with $\left| T_i(x)\right| ^{-1} \leq s_{1,i} < s_{2,i} $.
\end{proof}

\section{Finiteness of relative \texorpdfstring{$\Gamma$}{Gamma}-cohomology}
\label{subsec:pro-etale finiteness}

We now begin our treatment of higher direct images by applying the method of Cartan-Serre to get a finiteness result. Due to the unavoidably local nature of the construction of toric towers, some care is required with the homological algebra setup
(which we already took in \S\ref{S:totalization}). Note that this approach is not enough by itself to establish coherence of higher direct images, as it does not immediately establish compatibility with base change;
see \S\ref{sec:base change}. 

\begin{prop} \label{P:direct image finiteness}
Suppose that $K$ contains all $p$-power roots of unity.
Let $f: Y \to X$ be a smooth proper morphism of relative dimension $n$ of rigid analytic spaces over $K$.
Let $\calF$ be a coherent $(\varphi, \Gamma)$-module over $\tilde{\bA}_Y$ (resp.\ $\tilde{\bC}_Y$).
Then there exists an open covering $\{U_j\}_j$ of $X$ such that for any index $j$
and any perfectoid subdomain $\tilde{X}$ of $U_{j,\proet}$, the groups
$H^i((Y \times_X \tilde{X})_{\proet}, \calF|_{Y \times_X \tilde{X}})$ are pseudocoherent modules over
$\tilde{\bA}_{\tilde{X}}$ (resp.\ $\tilde{\bC}_{\tilde{X}}$).
Moreover, these groups vanish for $i > 2n$, and their formation is compatible with passage from $\tilde{X}$ to a pro-finite \'etale cover.
\end{prop}

Since the proof of Proposition~\ref{P:direct image finiteness} is rather involved, we break it up into a series of steps. We begin with some initial reductions.
(We distinguish the two phrasings in this statement, and all subsequent ones exhibiting a similar duality, as the \emph{type $\bA$ case} and the \emph{type $\bC$ case}.)

\begin{remark} \label{R:finiteness reduction}
In the type $\bC$ case of Proposition~\ref{P:direct image finiteness}, we may invoke 
Remark~\ref{R:type C interval} to replace $\calF$ with a pseudocoherent module
over $\tilde{\bC}^{[s,r]}_X$.
\end{remark}

We continue with a geometric reduction setting up a suitable framework for the Cartan-Serre method.
\begin{remark} \label{R:finiteness work locally}
By design, we are free to to work locally on $X$; 
that is, it suffices to show that for any fixed $x \in X$, we can replace $X$ with some neighborhood of $x$ for which the claim holds for the trivial open covering. This releases
the notation $U_i$ for another usage; see Definition~\ref{D:finiteness proper covering}.
\end{remark}

\begin{defn} \label{D:finiteness proper covering}
Using Lemma~\ref{L:proper covering} and Remark~\ref{R:finiteness work locally},
we may reduce to the case where $X = \Spa(A,A^+)$ is affinoid; $Y$ admits two coverings
$\{V_{1,i}\}_{i=1}^m, \{V_{2,i}\}_{i=1}^m$ with $V_{2,i} = \Spa(C_{2,i}, C_{2,i}^+) 
\subseteq V_{1,i} = \Spa(C_{1,i}, C_{1,i}^+)$; and for each $i$ there exists a commutative diagram
\begin{equation} \label{eq:proper covering2}
\xymatrix{
A\{\frac{s_{1,i,1}}{T_1},\dots,\frac{s_{1,i,n}}{T_n},\frac{T_1}{r_{1,i,1}},\dots,\frac{T_n}{r_{1,i,n}}\} \ar[d]
\ar[r] & A\{\frac{s_{2,i,1}}{T_1},\dots,\frac{s_{2,i,n}}{T_n},\frac{T_1}{r_{2,i,1}},\dots,\frac{T_n}{r_{2,i,n}}\} \ar[d] \\
C_{1,i} \ar[r] & C_{2,i}
}
\end{equation}
for some $0 < s_{1,i,j} < s_{2,i,j} \leq r_{2,i,j} < r_{1,i,j}$ in which the vertical arrows are finite \'etale morphisms of $A$-algebras.
We may further assume that for any nonempty subset $I \subseteq \{1,\dots,m\}$,
the spaces $V_{1,I} = \bigcap_{i \in I} V_{1,i}$, $V_{2,I} = \bigcap_{i \in I} V_{2,i}$
have the property that for each $i \in I$, $V_{1,I}$ is a rational subspace of $V_{1,i}$
and $V_{2,I}$ is a rational subspace of $V_{2,i}$.
\end{defn}

\begin{remark}
In the setting of Definition~\ref{D:finiteness proper covering}, 
we will prove that the claim holds on all of $X$.
To this end, we may assume without loss of generality that $\tilde{X} = \Spa(\tilde{A}, \tilde{A}^+)$
is the total space of a finite \'etale perfectoid tower $\psi_0$ over $X$
(otherwise, there is a nonfinite \'etale morphism at the bottom of the tower,
but we can pull back the data of Definition~\ref{D:finiteness proper covering}
along that \'etale morphism).
Note that the claimed compatibility with passage from $\tilde{X}$ to a pro-finite \'etale cover $\tilde{X}'$ follows from the fact that $\tilde{\bA}_{\tilde{X}'}/(p^n)$ (resp.\ $\tilde{\bC}^{[s,r]}_{\tilde{X}'}$)
is a topological direct sum of finite projective $\tilde{\bA}_{\tilde{X}}/(p^n)$-modules (resp.\ $\tilde{\bC}^{[s,r]}_{\tilde{X}}$-modules)
\cite[Remark~4.2.4]{part2}.
\end{remark}

We now build a particularly useful set of finite \'etale towers over subspaces of $Y$.
\begin{defn}
For $* = 1,2$, put
\[
(\tilde{C}_{1,i}, \tilde{C}_{1,i}^+) = (C_{1,i}, C_{1,i}^+) \widehat{\otimes}_{(A,A^+)} (\tilde{A}, \tilde{A}^+).
\]
Let $\psi_{*,i}$ denote the tower of finite \'etale coverings
$\{\Spa(\tilde{C}_{*,i,k},\tilde{C}_{*,i,k}^+) \to \Spa(\tilde{C}_{*,i}, \tilde{C}_{*,i}^+)\}_{k=0}^\infty$
in which
\[
\tilde{C}_{*,i,k} = \tilde{C}_{*,i}[T_1^{1/p^k}, \dots, T_n^{1/p^k}].
\]
Note that $\psi_{*,i}$ is a \emph{restricted toric tower} in the sense of
\cite[Definition~7.1.4]{part2}; according to \cite[Theorem~7.1.9]{part2},
all of the results of \cite[\S 5]{part2} apply to such towers.t
\end{defn}

We next describe the desired cohomology groups explicitly in terms of the specified towers.
\begin{defn} \label{D:compute complex}
For $* = 1,2$ and $I \subseteq \{1,\dots,m\}$ nonempty, let $\psi'_{*, I}$ 
be the tower
whose $k$-th term is the fiber product over $Y$ of the 
spaces $\Spa(\tilde{C}_{*,i,k},\tilde{C}_{*,i,k}^+)$ for all $i \in I$.
Let $\tilde{M}^{\prime,\bullet}_{*,I}$ be the \v{C}ech complex
for the sheaf $\calF$ and the cover $\psi'_{*,I}$; that is,
$\tilde{M}^{\prime,i}_{*,I}$ is the set of sections of $\calF$ on the $(i+1)$-fold
fiber product of the total space of $\psi'_{*,I}$.

For $s = 1,\dots,m$, we may interpolate the complexes $\tilde{M}^{\prime,\bullet}_{*,I}$ 
for $I \subseteq \{1,\dots,s\}$ into an $(s+1)$-dimensional complex
$\tilde{M}^{\prime,\bullet}_{*,\bullet,s}$,
taking $\tilde{M}^{\prime,\bullet}_{*,I}$ to be the zero complex for $I = \emptyset$.
We then have a natural morphism $\tilde{M}^{\prime,\bullet}_{1,\bullet,s} \to \tilde{M}^{\prime,\bullet}_{2,\bullet,s}$.
By \cite[Theorem~5.7.11]{part2}, the total complex $\tilde{M}^{\prime,\bullet}_{1,\bullet,s}$ 
(resp.\ $\tilde{M}^{\prime,\bullet}_{2,\bullet,s}$)
computes the cohomology of $\calF$ on the inverse image 
of $U_1 \cup \cdots \cup U_s$ (resp.\ $V_1 \cup \cdots \cup V_s$) in $\tilde{X}$.
In particular, for $s=m$, 
both total complexes compute the cohomology groups $H^i(\tilde{X}_{\proet}, \calF)$,
so the morphism between them is a quasi-isomorphism of total complexes.
\end{defn}

We next produce two alternate models for the single complex $\tilde{M}^{\prime,\bullet}_{*,I}$.

\begin{defn}
For $I \subseteq \{1,\dots,m\}$ nonempty, let $i = i(I)$ be the largest element of $M$.
Let $\psi_{1,I}$ (resp.\ $\psi_{2,I}$) be the tower obtained from $\psi_{1,i}$ (resp.\ $\psi_{2,I}$) by pullback from $U_i$ (resp.\ $V_i$) to $U_I$ (resp.\ $V_I$).
Let $\tilde{M}^{\bullet}_{*,I}$ be the \v{C}ech complex
for the sheaf $\calF$ and the cover $\psi_{*,I}$.
There is a natural morphism $\tilde{M}^{\bullet}_{*,I} \to \tilde{M}^{\prime, \bullet}_{*,I}$ of complexes which is a quasi-isomorphism \cite[Theorem~4.6.9]{part2}.

Let $\bA_{\psi_{*,I}}$, $\bC^{[s,r]}_{\psi_{*,I}}$ be the \emph{imperfect period rings}
associated to the tower $\psi_{*,I}$ as per \cite[Definition~5.2.1]{part2}.
Note that the morphism
$\bA_{\psi_{1,I}}/(p) \to \bA_{\psi_{2,I}}/(p)$
(resp. $\bC^{[s,r]}_{\psi_{1,I}} \to \bC^{[s,r]}_{\psi_{2,I}}$)
of Banach algebras over $\tilde{\bA}_{\tilde{X}}/(p)$ (resp.\
$\tilde{\bC}^{[s,r]}_{\tilde{X}}$) are completely continuous as 
morphisms of Banach modules.

In case (a) (resp.\ case (b)) of Proposition~\ref{P:direct image finiteness},
we may apply \cite[Theorem~5.8.16]{part2} (resp.\ \cite[Theorem~5.9.4]{part2}) to realize
$\calF|_{U_I}, \calF|_{V_I}$ as pseudocoherent $(\varphi, \Gamma)$-modules
over $\bA_{\psi_{1,I}}, \bA_{\psi_{2,I}}$
(resp.\ over $\bC^{[s,r]}_{\psi_{1,I}}, \bC^{[s,r]}_{\psi_{2,I}}$).
Let $M^\bullet_{*,I}$ be the resulting complex of continuous $\Gamma$-cochains for
$\Gamma \cong \ZZ_p^n$ the automorphism group of the tower $\psi_{*,I}$
(which is Galois because we forced $\tilde{A}$ to contain all $p$-power roots of unity).
There is a natural morphism $M^\bullet_{*,I} \to \tilde{M}^\bullet_{*,I}$ of complexes which is a quasi-isomorphism \cite[Theorem~5.8.15]{part2}, and the diagram
\begin{equation} \label{eq:cosplitting compatible}
\xymatrix{
M^\bullet_{1,I} \ar[r] \ar[d] & \tilde{M}^\bullet_{1,I} \ar[d] \ar[r] & \tilde{M}^{\prime, \bullet}_{1,I} \ar[d] \\
M^\bullet_{2,I} \ar[r] & \tilde{M}^\bullet_{2,I} \ar[r] & \tilde{M}^{\prime, \bullet}_{2,I}
}
\end{equation}
commutes.
\end{defn}

We next define splittings associated to the quasi-isomorphisms we have just introduced.
\begin{defn}
As in Definition~\ref{D:finiteness proper covering},
we may use \cite[Remark 4.2.4]{part2} to construct a morphism $\sigma'_{*,I}: \tilde{M}^{\prime, \bullet}_{*,I} \to \tilde{M}^{\bullet}_{*,I}$ splitting the given quasi-isomorphism in the other direction; more precisely, it is an inverse in the homotopy category to the given quasi-isomorphism. 

Since $\psi_{*,I}$ is a toric tower,
the map 
$\bA_{\psi_{*,I}}/(p) \to \tilde{\bA}_{\psi_{*,I}}/(p)$
(resp.\ $\bC^{[s,r]}_{\psi_{*,I}} \to \tilde{\bC}^{[s,r]}_{\psi_{*,I}}$)
admits a unique $\Gamma$-equivariant splitting induced by projection onto monomials in the $T_i$ with integral exponents, which is again an inverse in the homotopy category (compare \cite[Remark~7.1.11]{part2}).
Using such splittings, we obtain a morphism $\sigma_{*,I}: \tilde{M}^\bullet_{*,I} \to M^\bullet_{*,I}$.

In both cases, the splittings may (and will) be chosen compatibly with enlarging $I$ \emph{without changing $i(I)$}, by localizing the splitting constructed for $I = \{i(I)\}$.
Also, the splittings for $*=2$ may (and will) be taken to be base extensions of the splittings for $*=1$, so that the diagram
\begin{equation} \label{eq:splitting compatible}
\xymatrix{
\tilde{M}^{\prime,\bullet}_{1,I} \ar^{\sigma'_{1,I}}[r] \ar[d] & 
\tilde{M}^{\bullet}_{1,I} \ar^{\sigma_{1,I}}[r] \ar[d] &
M^\bullet_{1,I} \ar[d] \\
\tilde{M}^{\prime,\bullet}_{2,I} \ar^{\sigma'_{2,I}}[r] & 
\tilde{M}^{\bullet}_{2,I} \ar^{\sigma_{2,I}}[r] &
M^\bullet_{2,I}
}
\end{equation}
commutes.
 However, we cannot ensure any compatibility with enlarging $I$ when this changes the value of $i(I)$; this gives rise to some of the complexity in the ensuing arguments.
\end{defn}

We next use the splittings to interpolate the model complexes $M^\bullet_{*,I}$.
\begin{defn} \label{D:interpolate}
For $s \in \{1,\dots,m\}$, let $\tilde{N}^{\prime, \bullet}_{*,\bullet,s}$
be the subcomplex of $\tilde{M}^{\prime, \bullet}_{*,\bullet,s}$
composed of those terms for which the index $I$ satisfies $i(I) = s$.
Let $N^\bullet_{*,\bullet,s}$ be the subcomplex of $\tilde{N}^{\prime, \bullet}_{*,\bullet,s}$ defined by the inclusions $M^\bullet_{*,I} \to \tilde{M}^{\prime,\bullet}_{*,I}$ for all $I$ with $i(I) = s$; the fact that these inclusions define a subcomplex is due to the uniform derivation of the complexes $M^\bullet_{*,I}$ from the single tower $\psi_{*,s}$. For the same reason, the composition $\sigma_{*,\bullet} \circ \sigma'_{*,\bullet}$ defines a splitting $\tilde{N}^{\prime, \bullet}_{*,\bullet,s} \to N^{\bullet}_{*,\bullet,s}$ of the inclusion map, which provides an inverse in the homotopy category.

We may now view $\tilde{M}^{\prime, \bullet}_{*,\bullet,m}$ as the total complex associated to a sequence of morphisms
\[
\tilde{N}^{\prime, \bullet}_{*, \bullet,1} \to \cdots \to 
\tilde{N}^{\prime, \bullet}_{*, \bullet,m},
\]
where each term $\tilde{N}^{\prime, \bullet}_{*,\bullet,s}$ is isomorphic in the homotopy category to $N^{\bullet}_{*,\bullet,s}$. Moreover, the comparison map
$\tilde{N}^{\prime,\bullet}_{1,\bullet,s} \to \tilde{N}^{\prime, \bullet}_{2, \bullet,s}$
is represented in the homotopy category by the base extension morphism
$N^{\bullet}_{1,\bullet,s} \to N^{\bullet}_{2,\bullet,s}$, which 
by Remark~\ref{R:fingen cc} is completely continuous as long as this is true of the underlying ring homomorphism.
\end{defn}

\begin{remark} \label{R:reduce to fg}
We will prove below that the cohomology groups in  Proposition~\ref{P:direct image finiteness} are pseudocoherent. Concerning other assertions of the proposition,
we have already addressed compatibility with pro-finite \'etale covers
(see Definition~\ref{D:finiteness proper covering}).
To obtain vanishing in degrees above $2n$, it suffices to combine the following two observations: each individual
complex $M^\bullet_{*,I}$ has no cohomology above degree $n$;
and (if $X$ is quasicompact) the space $Y$ has cohomological dimension $\leq n$
\cite[Proposition~2.5.8]{dejong-vanderput}.
\end{remark}

We are now ready to prove Proposition~\ref{P:direct image finiteness}, starting with the type $\bA$ case.
\begin{proof}[Proof of Proposition~\ref{P:direct image finiteness}, type $\bA$ case]
As per Remark~\ref{R:reduce to fg}, it suffices to check that the cohomology groups are pseudocoherent, or even just finitely generated in light of Theorem~\ref{T:finitely generated to pseudocoherent type A}.

In the type $\bA$ case, the setup culminating in Definition~\ref{D:interpolate} is immediately susceptible to Lemma~\ref{L:Cartan-Serre homotopy} only when $\calF$ is killed by $p$,
because $\tilde{\bA}_{\tilde{X}}/(p^n)$ only admits the structure of a Banach ring for $n=1$. 
The commutativity conditions are all satisfied: the conditions in (a) hold because
$\tilde{N}^{\prime,\bullet}_{1,\bullet,s} \to \tilde{N}^{\prime, \bullet}_{2, \bullet,s}$ is a genuine morphism of genuine complexes (so in particular we may take all homotopies to be zero), whereas the conditions in (b) hold thanks to the commutativity of the diagrams
\eqref{eq:cosplitting compatible}, \eqref{eq:splitting compatible}.

In light of the previous discussion, when $\calF$ is killed by $p$, 
we may apply Lemma~\ref{L:Cartan-Serre homotopy} to deduce that the cohomology groups are contained in finitely generated modules, then \cite[Corollary~8.2.15]{part2} to deduce that they are finitely generated.
For general $\calF$, we use the $p$-torsion case to  deduce finite generation of the cohomology groups using Lemma~\ref{L:cohomology mod p}, taking $R=\tilde{\bA}_{\tilde{X}}$ and $\calA$ to be the category of $\tilde{\bA}_{\tilde{X}}$-modules equipped with commuting $\varphi^{\pm1},\Gamma$-actions. Note that the assumption on $\calA$ is ensured by \cite[Proposition 8.3.9]{part2} and Theorem~\ref{T:finitely generated to pseudocoherent type A}. 
\end{proof}

We next treat the type $\bC$ case.

\begin{proof}[Proof of Proposition~\ref{P:direct image finiteness}, type $\bC$ case]
As per Remark~\ref{R:reduce to fg}, it suffices to check that the cohomology groups are pseudocoherent, or even just finitely generated in light of 
Theorem~\ref{T:finitely generated to pseudocoherent}.
In the type $\bC$ case, the setup culminating in Definition~\ref{D:interpolate} is immediately susceptible to Lemma~\ref{L:Cartan-Serre homotopy}, with the same comments about the commutativity conditions; it thus follows that the cohomology groups are complete and contained in finitely generated modules. To promote this containment, we work through several layers of generality.
\begin{enumerate}
\item[(i)]
For $\calF$ killed by $t_\theta$, we proceed primarily by noetherian induction on the support of $\calF$, and secondarily by descending induction on cohomological degree.
Suppose that $\calF$ is killed by some ideal $J$ of $A$.
To establish finiteness of $h^i(\tilde{M}^{\prime, \bullet}_{*,\bullet,m})$ for some $i$
(given the same for all larger $i$), we work in the context of \cite[Remark~8.3.10]{part2}
(except that the letter $f$ is presently in use as a morphism, so we use $g$ for the ring element denoted $f$ in \textit{loc cit.}).

Let $Z_i$ and $Y_i$ be the modules of $i$-cocycles and $i$-coboundaries in 
$\tilde{M}^{\prime, \bullet}_{*,\bullet,m}$, so that the desired cohomology group is $M_i = Z_i/Y_i$. As described in \cite[Remark~8.3.10]{part2}, we can find $g \in A$ which is not a zero-divisor in $A/J$ or $\tilde{A}/J\tilde{A}$, such that $Z_i/Y_i$ admits a homomorphism to a finitely generated $\tilde{A}/J\tilde{A}$-module which becomes a split inclusion after inverting $g$; this then implies that the quotient of $M_i$ by its $g$-power-torsion submodule is pseudocoherent (using the induction hypotheses to verify conditions (a) and (b) of \cite[Remark~8.3.10]{part2}). 

Now suppose $M_i$ is not pseudocoherent, then the sequence of $g^n$-torsion submodules of $M_i$ does not stabilize as every $M_i[g^n]$ is pseudocoherent, and the union of this sequence contains the closure of $0$ in $M_i$. Moreover, the closure of $0$ in $M_i$ is not contained in any $M_i[g^n]$.  Otherwise, let $N_i$ be the quotient of $M_i$ by the closure of $0$.  Since $N_i$ is a Banach module, we may apply \cite[Remark~8.3.10]{part2} to deduce that $N_i[g^\infty]=N_i[g^n]$ for some $n$ because an infinite ascending union of Banach modules cannot be a Banach module unless the union stabilizes at some finite stage.  Therefore, we may conclude that the $g$-power-torsion submodules of $M_i$ stabilizes at some finite stage, yielding a contradiction.

Now repeat the construction after performing a base extension from $A$ to $A\{T\}$; let $Z'_i$ and $Y'_i$ be the resulting modules of $i$-cocycles and $i$-coboundaries, and let $M'_i = Z'_i/Y'_i$ be the quotient. Again, the closure of 0 in $M'_i$ consists entirely of $g$-power-torsion elements; however, if $M_i$ is not pseudocoherent, then for each sufficiently large $n$ there must exist some $z_n \in Z_i$ in the closure of $Y_i$ whose image in $M_i$ is killed by $g^n$ but not by $g^{n-1}$. By scaling each $z_n$ suitably, we may ensure that the sum $\sum_{n=0}^\infty z_n T^n$ converges to an element $z \in Z'_i$ in the closure of $Y'_i$ whose image in $M'_i$ is not killed by any power of $g$, a contradiction. This proves the desired finiteness result.
  
\item[(ii)]
For $\calF$ killed by some $t \in \tilde{\bC}^{[s,r]}_K$ coprime to $t_\theta$, 
\cite[Corollary~8.8.10]{part2} (as applied in Remark~\ref{R:coherent type C})
implies the desired finite generation.

\item[(iii)]
Using (i) and (ii), we deduce the claim whenever $\calF$ is killed by any nonzero
$t \in \tilde{\bC}^{[s,r]}_K$; 
we may thus argue as in (i), using some suitable 
$t \in \tilde{\bC}^{[s,r]}_K$ in place of $g$, to conclude.
\end{enumerate}

The proof is complete.
\end{proof}

\section{Base change}
\label{sec:base change}

As noted previously, to establish coherence of higher direct images, one must combine Proposition~\ref{P:direct image finiteness} with a certain compatibility with base change. We address this point next.

\begin{prop} \label{P:base change}
With hypotheses and notation as in Proposition~\ref{P:direct image finiteness},
for any index $j$, any perfectoid subdomain $\tilde{X}$ of $U_{j,\proet}$,
and any perfectoid subdomain $\tilde{X}'$ of $\tilde{X}_{\proet}$, the morphisms
\begin{equation} \label{eq:base change1}
H^i(\tilde{X}_{\proet}, \calF) \otimes_{\tilde{\bA}_{\tilde{X}}} \tilde{\bA}_{\tilde{X}'}
\to H^i(\tilde{X}'_{\proet}, \calF)
\quad
(\mbox{resp. }
H^i(\tilde{X}_{\proet}, \calF) \otimes_{\tilde{\bC}^{[s,r]}_{\tilde{X}}} \tilde{\bC}^{[s,r]}_{\tilde{X}'}
\to H^i(\tilde{X}'_{\proet}, \calF)
)
\end{equation}
are isomorphisms for all $i \geq 0$.
\end{prop}

Again, we devote the entire remainder of this section to the proof of Proposition~\ref{P:base change},
and thus retain notation as therein.

\begin{remark} \label{R:cohomology pro-finite etale}
To begin, note that if $\tilde{X}'$ is a pro-finite \'etale cover of $\tilde{X}$, then
the claimed compatibility (i.e., the fact that the maps in \eqref{eq:base change1}
are isomorphisms) is already included in Proposition~\ref{P:direct image finiteness}.
\end{remark}

\begin{remark} \label{R:completion preserves cohomology}
From the proof of Proposition~\ref{P:direct image finiteness},
we obtain a bounded complex $\tilde{M}^\bullet$ such that
$H^i(\tilde{X}_{\proet}, \calF) \cong h^i(\tilde{M}^\bullet)$
(e.g., the totalization of the complex denoted $\tilde{M}^{\prime,\bullet}_{1,\bullet,m}$
in Definition~\ref{D:compute complex}).
Define the second complex 
\[
\tilde{M}^{\star \bullet} = \tilde{M}^\bullet \otimes_{\tilde{\bA}_{\tilde{X}}} \tilde{\bA}_{\tilde{X}'} 
\qquad
(\mbox{resp. } \tilde{M}^{\star \bullet} = \tilde{M}^\bullet \otimes_{\tilde{\bC}^{[s,r]}_{\tilde{X}}} \tilde{\bC}^{[s,r]}_{\tilde{X}'}).
\]
Define the third complex
\[
\tilde{M}^{\prime \bullet} = \tilde{M}^\bullet \widehat{\otimes}_{\tilde{\bA}_{\tilde{X}}} \tilde{\bA}_{\tilde{X}'}
\qquad
(\mbox{resp. } \tilde{M}^{\prime \bullet} \widehat{\otimes}_{\tilde{\bC}^{[s,r]}_{\tilde{X}}} \tilde{\bC}^{[s,r]}_{\tilde{X}'});
\]
then $H^i(\tilde{X}'_{\proet}, \calF) \cong h^i(\tilde{M}^{\prime \bullet})$.

By hypothesis,
we know that the $h^i(\tilde{M}^{\bullet})$
are pseudocoherent modules which vanish in large enough degrees.
By Remark~\ref{R:pseudoflat},
we have isomorphisms
\begin{equation} \label{eq:flat base change}
h^i(\tilde{M}^{\bullet}) \otimes_{\tilde{\bA}_{\tilde{X}}}
\tilde{\bA}_{\tilde{X}'} 
\cong
h^i(\tilde{M}^{\star \bullet})
\qquad
(\mbox{resp. }
h^i(\tilde{M}^\bullet) \otimes_{\tilde{\bC}^{[s,r]}_{\tilde{X}}}
\tilde{\bC}^{[s,r]}_{\tilde{X}'} 
\cong
h^i(\tilde{M}^{\star \bullet})
);
\end{equation}
the content of Proposition~\ref{P:base change} is thus that the morphism $\tilde{M}^{* \bullet} \to \tilde{M}^{\prime \bullet}$ is a quasi-isomorphism.
\end{remark}

\begin{lemma} \label{L:completion preserves cohomology2}
Suppose that there exists $f \in \calO(X)$ such that 
$\tilde{X}' = \{v \in \tilde{X}: v(f) \geq 1\}$.
Then the morphisms in \eqref{eq:base change1} are isomorphisms.
\end{lemma}
\begin{proof}
Retain notation as in Remark~\ref{R:completion preserves cohomology}.
By approximating $f$ suitably as in \cite[Corollary~3.6.7]{part1},
we can find $\overline{f} \in \overline{\calO}(\tilde{X})$ such that
$\tilde{X}' = \{v \in \tilde{X}: v(\theta([\overline{f}])) \geq 1\}$.
Put $\tilde{A} = \tilde{\bA}_{\tilde{X}}, \tilde{A}' = \tilde{\bA}_{\tilde{X}'}$
(resp.\ $\tilde{A} = \tilde{\bC}^{[s,r]}_{\tilde{X}}, \tilde{A}' = \tilde{\bC}^{[s,r]}_{\tilde{X}'}$);  we then have the exact sequence
\begin{equation} \label{eq:invert base extension}
0 \to \tilde{A}\{T\} \stackrel{\times 1-[\overline{f}] T}{\longrightarrow} \tilde{A}\{T\} \to \tilde{A}' \to 0.
\end{equation}
Similarly, for any perfectoid subdomain $\tilde{U}$ of $Y$ lying over $\tilde{X}$,
if we put $\tilde{U}' = \tilde{U} \times_{\tilde{X}} \tilde{X}'$
and $\tilde{B} = \tilde{\bA}_{\tilde{U}}, \tilde{B}' = \tilde{\bA}_{\tilde{U}'}$
(resp.\ $\tilde{B} = \tilde{\bC}^{[s,r]}_{\tilde{U}}, \tilde{B}' = \tilde{\bC}^{[s,r]}_{\tilde{U}'}$);  we have an exact sequence
\begin{equation} \label{eq:invert base extension2}
0 \to \tilde{B}\{T\} \stackrel{\times 1-[\overline{f}] T}{\longrightarrow} \tilde{B}\{T\} \to \tilde{B}' \to 0.
\end{equation}
Put
\begin{gather*}
\tilde{Z}^i = \ker(\tilde{M}^i \to \tilde{M}^{i+1}), \qquad
\tilde{Y}^i = \image(\tilde{M}^{i-1} \to \tilde{M}^{i}), \\
\tilde{Z}^{\prime i} = \ker(\tilde{M}^{\prime i} \to \tilde{M}^{\prime i+1}), \qquad
\tilde{Y}^{\prime i} = \image(\tilde{M}^{\prime i-1} \to \tilde{M}^{\prime i}).
\end{gather*}
For $* = \tilde{M}^i, \tilde{Y}^i, \tilde{Z}^i, h^i(\tilde{M}^\bullet)$,
let $*\{T\}$ be the completion of $*[T]$ (i.e., the module $* \otimes_A A[T]$
viewed as polynomials with coefficients in $*$) for the Gauss norm. For each $i$, the sequence
\[
0 \to \tilde{M}^i\{T\} \stackrel{\times 1-[\overline{f}]T}{\longrightarrow} \tilde{M}^i\{T\} \to \tilde{M}^{\prime i} \to 0
\]
is exact: exactness at the middle and right follow by (uncompleted) base extension from \eqref{eq:invert base extension2}, and exactness at the left follows by considering formal power series in $T$ (compare \cite[Remark~1.2.5]{part2}).

We prove by descending induction on $i$ that the sequences
\begin{gather}
\label{eq:completion cohomology1}
0 \to \tilde{Y}^i\{T\}\stackrel{\times 1-[\overline{f}]T}{\longrightarrow} \tilde{Y}^i\{T\} \to \tilde{Y}^{\prime i} \to 0 \\
\label{eq:completion cohomology2}
0 \to \tilde{Z}^i\{T\}\stackrel{\times 1-[\overline{f}]T}{\longrightarrow} \tilde{Z}^i\{T\} \to \tilde{Z}^{\prime i} \to 0 \\
\label{eq:completion cohomology3}
0 \to h^i(\tilde{M}^\bullet)\{T\}\stackrel{\times 1-[\overline{f}]T}{\longrightarrow} h^i(\tilde{M}^\bullet)\{T\} \to h^i(\tilde{M}^{\prime \bullet}) \to 0
\end{gather}
are exact.
Given \eqref{eq:completion cohomology1} with $i$ replaced by $i+1$,
in the diagram
\begin{equation} \label{eq:completion cohomology diag1}
\xymatrix{
 & 0 \ar[d] & 0 \ar[d] & 0 \ar[d] & \\
0 \ar[r] & \tilde{Z}^i\{T\} \ar[r] \ar[d] & \tilde{M}^i\{T\} \ar[r] \ar[d] & \tilde{Y}^{i+1}\{T\} \ar[r] \ar[d] & 0 \\
0 \ar[r] & \tilde{Z}^i\{T\} \ar[r] \ar[d] & \tilde{M}^i\{T\} \ar[r] \ar[d] & \tilde{Y}^{i+1}\{T\} \ar[r] \ar[d] & 0 \\
0 \ar[r] & \tilde{Z}^{\prime i} \ar[r] \ar[d] & \tilde{M}^{\prime i} \ar[r] \ar[d] & \tilde{Y}^{\prime i+1} \ar[r] \ar[d] & 0 \\
& 0 & 0 & 0
}
\end{equation}
all three rows and the middle and right columns are exact; by the snake lemma, 
we infer that \eqref{eq:completion cohomology2} is exact.
In the diagram
\begin{equation} \label{eq:completion cohomology diag2}
\xymatrix{
 & 0 \ar[d] & 0 \ar[d] & 0 \ar[d] & \\
0 \ar[r] & \tilde{Y}^i\{T\} \ar[r] \ar[d] & \tilde{Z}^i\{T\} \ar[r] \ar[d] & h^i(\tilde{M}^\bullet)\{T\} \ar[r] \ar[d] & 0 \\
0 \ar[r] & \tilde{Y}^i\{T\} \ar[r] \ar[d] & \tilde{Z}^i\{T\} \ar[r] \ar[d] & h^i(\tilde{M}^\bullet)\{T\} \ar[r] \ar[d] & 0 \\
0 \ar[r] & \tilde{Y}^{\prime i} \ar[r] \ar[d] & \tilde{Z}^{\prime i} \ar[r] \ar[d] & h^i(\tilde{M}^{\prime \bullet}) \ar[r] \ar[d] & 0 \\
& 0 & 0 & 0
}
\end{equation}
all three rows and the middle column are exact.
In the right column, exactness at the top holds by \cite[Remark~1.2.5]{part2}
and exactness at the bottom holds because the bottom row and the middle column are exact.
By the snake lemma again, we deduce that the left column is exact at the top and middle.

We now peek ahead to the diagram \eqref{eq:completion cohomology diag1} with $i$ replaced by $i-1$;
again, all three rows and the middle column are exact, which forces the right column to be exact at the bottom as well as in the other two positions. Returning to \eqref{eq:completion cohomology diag2}, we now have exactness in all rows and the left and middle columns; by the snake lemma once more, we deduce that the right column is exact. To summarize, we now have exactness of all three equations
 \eqref{eq:completion cohomology1},  \eqref{eq:completion cohomology2},  \eqref{eq:completion cohomology3}, thus completing the induction.
 
From \eqref{eq:invert base extension} and \eqref{eq:completion cohomology3}, we see that
$h^i(\tilde{M}^{* \bullet}) \to h^i(\tilde{M}^{\prime \bullet})$ is an isomorphism for all $i$. This is the desired result.
\end{proof}

\begin{cor} \label{C:completion preserves cohomology}
There exists a neighborhood basis $\calB$ of $\tilde{X}_{\proet}$ such that for each
$\tilde{X}'\in \calB$, the morphisms \eqref{eq:base change1} are isomorphisms.
\end{cor}
\begin{proof}
Combine Lemma~\ref{L:completion preserves cohomology2} with \cite[Lemma~2.4.6]{part2}
(to get a neighborhood basis of $X_{\et}$ and hence of $\tilde{X}_{\et}$)
and Remark~\ref{R:cohomology pro-finite etale} (to enlarge this to a neighborhood basis of $\tilde{X}_{\proet}$).
\end{proof}

\begin{proof}[Proof of Proposition~\ref{P:base change}]
By Corollary~\ref{C:completion preserves cohomology},
there exists a covering $\{\tilde{U}_j\}$ of $\tilde{X}'$ such that in the composition
\begin{gather*}
\bigoplus_j H^i(\tilde{X}_{\proet}, \calF) \otimes_{\tilde{\bA}_{\tilde{X}}} \tilde{\bA}_{\tilde{U}_j}
\to
\bigoplus_j H^i(\tilde{X}'_{\proet}, \calF) \otimes_{\tilde{\bA}_{\tilde{X}'}} \tilde{\bA}_{\tilde{U}_j}
\to
\bigoplus_j H^i(\tilde{U}_{j,\proet}, \calF) \\
(\mbox{resp. } 
\bigoplus_k H^i(\tilde{X}_{\proet}, \calF) \otimes_{\tilde{\bC}^{[s,r]}_{\tilde{X}}} \tilde{\bC}^{[s,r]}_{\tilde{U}_j}
\to
\bigoplus_k H^i(\tilde{X}'_{\proet}, \calF) \otimes_{\tilde{\bC}^{[s,r]}_{\tilde{X}'}} \tilde{\bC}^{[s,r]}_{\tilde{U}_j}
\to
\bigoplus_k H^i(\tilde{U}_{j,\proet}, \calF) ),
\end{gather*}
both the second arrow and the composition are isomorphisms; consequently, the first
arrow is also an isomorphism. That is,
the morphisms in \eqref{eq:base change1}
become isomorphisms after tensoring over $\tilde{\bA}_{\tilde{X}}$
(resp.\ over $\tilde{\bC}^{[s,r]}_{\tilde{X}}$)
with $\bigoplus_j \tilde{\bA}_{\tilde{U}_j}$ (resp.\
over $\bigoplus_j \tilde{\bC}^{[s,r]}_{\tilde{X}}$).
Since this tensoring is exact on pseudocoherent modules
(by Remark~\ref{R:pseudoflat})
and faithful (by \cite[Lemma~2.3.12]{part1}), this implies the desired isomorphism.
\end{proof}

\section{Higher direct images and applications}
\label{sec:higher direct images}

With the preceding calculations in hand, we derive global conclusions about the relative cohomology of coherent $(\varphi, \Gamma)$-modules, and consequences for local systems.
\begin{theorem} \label{T:direct image}
Let $f: Y \to X$ be a smooth proper morphism of relative dimension $n$ of rigid analytic spaces over $K$. 
Let $\calF$ be a coherent $(\varphi, \Gamma)$-module over $\tilde{\bA}_Y$ (resp.\ $\tilde{\bC}^{[s,r]}_Y$).
Then the higher direct images $R^i f_{\proet *} \calF$ are coherent $(\varphi, \Gamma)$-modules
over $\tilde{\bA}_X$ (resp. $\tilde{\bC}^{[s,r]}_X$); moreover, these sheaves
vanish for $i > 2n$.
\end{theorem}
\begin{proof}
By combining Proposition~\ref{P:direct image finiteness}
and Proposition~\ref{P:base change}, we immediately obtain the desired conclusion.
\end{proof}

\begin{remark} \label{R:acyclicity}
In light of the acyclicity of pseudocoherent sheaves on affinoid perfectoid spaces
\cite[Corollary~3.5.6]{part2}, Theorem~\ref{T:direct image} implies \emph{a posteriori} that the assertions of Proposition~\ref{P:direct image finiteness} 
and Proposition~\ref{P:base change} both hold for the trivial covering of $X$, without any restriction on $K$. (Namely, compute the cohomology groups from the higher direct images using the Leray spectral sequence.)
\end{remark}

By specializing $X$ to a point, we obtain the following corollary.
\begin{theorem} \label{T:main finiteness}
Assume that $K$ is a finite extension of $\Qp$.
Let $X$ be a smooth proper rigid analytic variety of dimension $n$ over $K$.
Let $\calF$ be a coherent $(\varphi, \Gamma)$-module over $\tilde{\bA}_X$ (resp.\ $\tilde{\bC}_X$).
Then the cohomology groups $H^i_{\varphi,\Gamma}(\calF)$ are finite $\ZZ_p$-modules
(resp.\ finite-dimensional $\Qp$-vector spaces); moreover, these groups vanish for $i > 2n+2$.
\end{theorem}
\begin{proof}
This follows by combining 
Theorem~\ref{T:integral ordinary phi-Gamma finiteness}
(resp.\ Theorem~\ref{T:ordinary phi-Gamma finiteness}), \cite[Theorem 5.7.10]{part2}(resp.\ \cite[Theorem 5.7.11]{part2}), and 
Theorem~\ref{T:direct image}.
\end{proof}

By further specializing to \'etale $(\varphi, \Gamma)$-modules, we obtain the following corollary, which includes Theorem~\ref{T:finiteness1}(b).

\begin{theorem} \label{T:main finiteness etale}
Assume that $K$ is a finite extension of $\Qp$.
Let $X$ be a smooth proper rigid analytic variety over $K$.
\begin{enumerate}
\item[(a)]
Let $T$ be an \'etale $\Zp$-local system on $X$,
or more generally a locally finite $\underline{\Zp}$-module on $X_{\proet}$. Then the pro-\'etale cohomology groups
$H^i(X_{\proet}, T)$ are finite $\Zp$-modules.
\item[(b)]
Let $V$ be an \'etale $\Qp$-local system on $X$. Then the pro-\'etale cohomology groups
$H^i(X_{\proet}, V)$ are finite-dimensional $\Qp$-vector spaces.
\end{enumerate}
Moreover, in both cases, these groups vanish for $i > 2n+2$.
\end{theorem}
\begin{proof}
Both parts follow from combining Theorem~\ref{T:main finiteness} with
\cite[Theorem~9.4.5]{part1}.
\end{proof}

\begin{remark}
Theorem~\ref{T:main finiteness etale}(a) may also be deduced by reducing to the corresponding statement for an $\Fp$-local system, which is a finiteness theorem of
Scholze \cite[Theorem~5.1]{scholze2}. 
One can continue in this fashion to establish Theorem~\ref{T:main finiteness etale}(b) for isogeny $\Zp$-local systems, i.e., \'etale $\Qp$-local systems admitting a stable lattice.
However, it is unclear how to extend the methods used to prove \cite[Theorem~5.1]{scholze2} to establish Theorem~\ref{T:main finiteness etale}(b) for general $\Qp$-local systems, or for more general $(\varphi, \Gamma)$-modules.
\end{remark}

\section{Base change revisited}
\label{sec:base change2}

In this section, we establish the following general base change theorem. The statement
can be formally promoted to a derived version; see Remark~\ref{R:complexes}.
\begin{theorem} \label{T:general base change}
With notation as in Theorem~\ref{T:direct image}, let $g: X' \to X$ be any morphism of rigid spaces over $K$. Put $Y' = Y \times_X X'$ and let $f': Y' \to X'$, $g': Y' \to Y$ be the base extensions of $f,g$. Then the natural map
\begin{equation} \label{eq:base change general}
\LL g_{\proet}^* \RR f_{\proet *} \calF \to \RR f'_{\proet *} \LL g_{\proet}^{\prime *} \calF
\end{equation}
is an isomorphism.
\end{theorem}

\begin{remark} \label{R:general base change injective}
In the setting of Theorem~\ref{T:general base change},
let $\tilde{X}$ be a perfectoid subdomain of $X$, and
let $\tilde{M}^\bullet$ be a bounded complex computing $H^i(\tilde{X}_{\proet}, \calF)$.
Let $\tilde{X}'$ be a perfectoid subdomain of $X'$ mapping to $\tilde{X}$.
As in Remark~\ref{R:completion preserves cohomology}, define
the second complex 
\[
\tilde{M}^{\star \bullet} = \tilde{M}^\bullet \otimes_{\tilde{\bA}_{\tilde{X}}} \tilde{\bA}_{\tilde{X}'} 
\qquad
(\mbox{resp. }  \tilde{M}^{\star \bullet} = \tilde{M}^\bullet \otimes_{\tilde{\bC}^{[s,r]}_{\tilde{X}}} \tilde{\bC}^{[s,r]}_{\tilde{X}'})
\]
and the third complex
\[
\tilde{M}^{\prime \bullet} := \tilde{M}^\bullet \widehat{\otimes}_{\tilde{\bA}_{\tilde{X}}} \tilde{\bA}_{\tilde{X}'}
\qquad
(\mbox{resp. } \tilde{M}^{\prime \bullet} = \tilde{M}^\bullet \widehat{\otimes}_{\tilde{\bC}^{[s,r]}_{\tilde{X}}} \tilde{\bC}^{[s,r]}_{\tilde{X}'});
\]
then $H^i(\tilde{X}'_{\proet}, \LL g^{\prime *}_{\proet} \calF) \cong h^i(\tilde{M}^{\prime \bullet})$, and the claim is that the map $\tilde{M}^{\star \bullet} \to \tilde{M}^{\prime \bullet}$ is a quasi-isomorphism. 
\end{remark}

The following argument is inspired by \cite[Corollary~5.12]{scholze2},
specifically in its use of \cite[Proposition~2.6.1]{huber}; however, we need to recast the latter in terms of a suitable topology.
\begin{lemma} \label{L:base change point1}
In Theorem~\ref{T:general base change}, suppose that $\calF$ is killed by $p$ (resp.\ by some $t \in \tilde{\bC}^{[s,r]}_K$ coprime to $t_\theta$).
Then the map \eqref{eq:base change general} is an isomorphism.
\end{lemma}
\begin{proof}
As in \cite[\S 3.5]{part2},
define the \emph{v-topology} on the category of perfectoid spaces in which a covering is simply a surjective morphism for which the inverse image of each quasicompact open subset is contained in a quasicompact open subsets (by analogy with the h-topology for schemes).
This topology has the following properties (for $\tilde{X}$ a perfectoid space).
\begin{itemize}
\item
Any finite projective $\calO_{\tilde{X}}$-module is acyclic for the $v$-topology
\cite[Theorem~3.5.5]{part2}.
\item
Let $\nu$ be the morphism from the v-topology of $\tilde{X}$ to the analytic site. Then
pullback via $\nu$ defines an equivalence of categories between the categories of vector bundles for the analytic topology and the v-topology \cite[Theorem~3.5.8]{part2}.
\item
For $\calF$ a finite projective $\calO_{\tilde{X}}$-module for the v-topology,
we have $R^i \nu_* \calF = 0$ for $i>0$ \cite[Corollary~3.5.9]{part2}. In particular,
for $\calF$ a finite projective $\calO_{\tilde{X}}$-module for the analytic topology,
the map $\calF \to \RR \nu_* (\nu^* \calF)$ is an isomorphism.
\end{itemize}
Using the perfectoid correspondence, we deduce the corresponding assertions with $\calO$ replaced by $\tilde{\bA}/(p)$ or $\tilde{\bC}^{[s,r]}/(t)$.

By Remark~\ref{R:coherent type A} (resp.\ Remark~\ref{R:coherent type C}),
$\calF$ is finite projective over $\tilde{\bA}_Y/(p)$
(resp.\ over $\tilde{\bC}^{[s,r]}_Y/(t)$). 
Similarly, by Proposition~\ref{P:direct image finiteness},
the sheaves $R^i f_{\proet *} \calF$ are finite
projective over $\tilde{\bA}_X/(p)$ (resp.\ over $\tilde{\bC}^{[s,r]}_X/(t)$).
In particular, the ordinary and derived pullbacks of these sheaves along any morphism of perfectoid spaces coincide.

In light of the previous discussion, we may argue as follows.
Let $\tilde{X}$ be a perfectoid subdomain of $X$, let $\tilde{Y}$ be a perfectoid subdomain of $Y$ lying over $\tilde{X}$, let $\tilde{X}'$ be a perfectoid subdomain of $X'$ lying over 
$\tilde{X} \times_X X'$ (which may not itself be perfectoid), and let $\tilde{Y}'$ be a perfectoid subdomain of $Y'$ lying over $\tilde{X}' \times_X Y$. 
If we restrict the morphisms $f,g, f', g'$ to these perfectoid subdomains,
then the morphism in \eqref{eq:base change general} is naturally isomorphic to the same morphism computed using the v-topology instead of the pro-\'etale topology,
and the latter is an isomorphism because \emph{every} morphism of perfectoid spaces is part of a v-covering. This proves the claim.
\end{proof}

\begin{lemma} \label{L:base change point2}
In Theorem~\ref{T:general base change}, suppose that $X'$ is a point and that $\calF$ is
killed by $t_\theta$.
Then the map \eqref{eq:base change general} is an isomorphism.
\end{lemma}
\begin{proof}
Since the formation of both sides commutes with extension of $K$, we may reduce to the case where $X'$ is a $K$-rational point. 
In the notation of Remark~\ref{R:general base change injective},
we must check that the morphism from the uncompleted base extension
$\tilde{M}^{* \bullet}$ to the completed base extension $\tilde{M}^{\prime \bullet}$ is a quasi-isomorphism. However, in this case $\calF$ is a coherent $\widehat{\calO}_Y$-module and the morphisms 
\[
\widehat{\calO}_X \to g_{\proet *} \widehat{\calO}_{X'},
\qquad
\widehat{\calO}_Y \to g'_{\proet *} \widehat{\calO}_{Y'}
\]
are surjective, so in fact each individual morphism $\tilde{M}^{* i} \to \tilde{M}^{\prime i}$ is an isomorphism.
\end{proof}

The following argument emulates \cite[Lemma~4.1.5]{kpx}
\begin{proof}[Proof of Theorem~\ref{T:general base change}]
We wish to check that that
\eqref{eq:base change general} induces isomorphisms of cohomology groups,
or equivalently that the associated mapping cone $C = \Cone(\tilde{M}^* \to \tilde{M}')$
is acyclic. We know that the cohomology groups of $\tilde{M}'$ are pseudocoherent,
and that the cohomology groups of $\tilde{M}^*$ are the base extensions of pseudocoherent modules. Using Remark~\ref{R:coherent type A} in the type $\bA$ case, or
Theorem~\ref{T:pseudocoherent noetherian} in the type $\bC$ case, it follows that the cohomology groups of $C$ are coherent $(\varphi, \Gamma)$-modules. If one of them were nonzero, then this could be detected from its reduction modulo $p$ (in the type $\bA$ case)
or some irreducible $t \in \tilde{\bC}^{[s,r]}_K$ (in the type $\bC$ case) evaluated at a suitable point; however, this is ruled out by 
Lemma~\ref{L:base change point1} and Lemma~\ref{L:base change point2}.
\end{proof}

We conclude this discussion with a formal observation.
\begin{remark} \label{R:complexes}
Since coherent $(\varphi, \Gamma)$-modules form an abelian category and are preserved under derived pullback along an arbitrary morphism
(by Remark~\ref{R:coherent type A} in the type $\bA$ case and
Theorem~\ref{T:pseudocoherent noetherian} in the type $\bC$ case)
and derived pushforward along a smooth proper morphism (by Theorem~\ref{T:direct image}),
Theorem~\ref{T:general base change} as stated
immediately implies the corresponding statement with the sheaf $\calF$ replaced by a bounded complex of coherent $(\varphi, \Gamma)$-modules,
or more generally a perfect complex of sheaves of $\tilde{\bA}$-modules
(resp.\ $\tilde{\bC}^{[s,r]}$-modules) whose cohomology groups are coherent
$(\varphi, \Gamma)$-modules. Note that in the latter, it is not necessary to equip the complex itself with an action of $\varphi$. Similar observations apply to our other  results.
\end{remark}

\section{Projectivity of higher direct images}

In Theorem~\ref{T:direct image}, if $\calF$ is a projective $(\varphi, \Gamma)$-module,
it is not apparent under what conditions its higher direct images will themselves be projective, rather than merely coherent.
We establish some sufficient conditions for this, starting with the type $\bA$ case.

\begin{theorem} \label{T:etale direct images type A}
With notation as in Theorem~\ref{T:direct image}, suppose that 
$\calF$ is the coherent $(\varphi, \Gamma)$-module
over $\tilde{\bA}_Y$
associated to the locally finite $\underline{\ZZ_p}$-module $T$ on $Y$.
Then the higher direct images of $\calF$ along $Y_{\proet} \to X_{\proet}$ are the coherent $(\varphi, \Gamma)$-modules over $\tilde{\bA}_X$ associated to the higher direct images of $T$; in particular, the higher direct images of $\calF$ are projective
if and only if the higher direct images of $T$ are $p$-torsion-free (i.e., are $\ZZ_p$-local systems).
\end{theorem}
\begin{proof}
By the second part of Theorem~\ref{T:Artin-Schreier},
we may recover the higher direct images of $T$ as the cohomology groups of the mapping cone of $\varphi-1$ on the derived pushforward of $\calF$.
Now note that for any $\varphi$-module over $\tilde{\bA}_X$, the map $\varphi-1$ is surjective as a morphism of pro-\'etale sheaves; we may thus
recover the higher direct images of $T$
from the individual higher direct images of $\calF$ by taking $\varphi$-invariants.
By the first part of Theorem~\ref{T:Artin-Schreier}, we may turn around and recover the higher direct images of $\calF$ from the higher direct images of $T$ by tensoring over $\underline{\ZZ_p}$ with $\tilde{\bA}_X$. This proves the claim.
\end{proof}

\begin{remark}
It is possible for the higher direct images of $T$ to have nontrivial $p$-torsion even if $T$ is $p$-torsion-free. See for example \cite[\S 2.2]{bms}.
\end{remark}

\begin{theorem} \label{T:main finiteness etale alg closed type A}
Assume that $K$ is algebraically closed.
Let $X$ be a smooth proper rigid analytic variety over $K$.
Let $T$ be an \'etale $\Zp$-local system on $X$. Then the pro-\'etale cohomology groups
$H^i(X_{\proet}, T)$ are finite $\Zp$-modules;
moreover, these groups vanish for $i > 2n$.
\end{theorem}
\begin{proof}
This is immediate from 
Theorem~\ref{T:direct image} plus Theorem~\ref{T:etale direct images type A}.
\end{proof}

\begin{remark}
The analogue of Theorem~\ref{T:etale direct images type A}
in type $\bC$ is subtler because the functor from $\QQ_p$-local systems to $(\varphi, \Gamma)$-modules is only fully faithful, not essentially surjective; moreover, the surjectivity of $\varphi-1$ is true on the essential image of this functor, but not on the whole category. We thus can only say \emph{a priori} that there is a natural quasi-isomorphism
\begin{equation} \label{eq:etale direct images derived}
\RR f_* V \cong \Cone(\varphi-1, \RR f_* \calF).
\end{equation}
In particular, there is a natural morphism
\begin{equation} \label{eq:etale direct images}
(R^i f_* V) \otimes_{\underline{\QQ_p}} \tilde{\bC}^{[s,r]}_X \to R^i f_* \calF
\end{equation}
but it is not guaranteed to be either injective or surjective.
\end{remark}

To address the previous discussion, we first treat the absolute case (i.e., where the base space is a point) in detail.

\begin{lemma} \label{L:cohomology base field}
Suppose that $K$ is algebraically closed.
Let $X = \Spa(A,A^+)$ be a smooth affinoid space over $K$. Let $T$ be a locally finite $\underline{\ZZ_p}$-module on $X$. Then for all $i \geq 0$, the following statements hold.
\begin{enumerate}
\item[(a)]
The group $H^i(X_{\proet}, T)$ is a finite $\ZZ_p$-module.
\item[(b)]
 Let $K'$ be an algebraically closed nonarchimedean field containing $K$, put $X' = X \times_K K'$,
and let $T'$ be the pullback of $T$ to $X'$. Then the morphism
$H^i(X_{\proet}, T) \to H^i(X'_{\proet}, T')$ is an isomorphism.
\end{enumerate}
\end{lemma}
\begin{proof}
Since both claims are local, we may assume that $X$ is the base of a restricted toric tower $\psi$. Let $\psi'$ be the base extension of this tower to $X'$.
Let $\bA_\psi$ (resp.\ $\bA_{\psi'}$) be the imperfect period ring associated to the tower $\psi$ (resp.\ $\psi'$); by \cite[Theorem~7.1.5]{part2}, the vertical arrows in the commutative diagram
\[
\xymatrix{
H^i(X_{\proet}, T) \ar[r] \ar[d] & H^i(X_{\proet}, T') \ar[d] \\
H^i_{\varphi, \Gamma}(\calF) \ar[r] & H^i_{\varphi, \Gamma}( \calF')
}
\]
are isomorphisms. To check that the horizontal arrows are isomorphisms, by Lemma~\ref{L:cohomology mod p} it suffices to check the case where $T$ is killed by $p$; using the adjunction morphism for a suitable finite \'etale cover of $X$,
we may further assume that $T$ is trivial and $X$ is connected.
In this case, let $R_\psi$ (resp.\ $R_{\psi'})$ denote the reduction of $\bA_\psi$ (resp.\ $\bA_{\psi'}$) modulo $p$. Let $K^\flat$ (resp.\ $K^{\prime\flat}$) denote the tilt of $K$ (resp.\ $K'$). The residue map
\[
f\mapsto \Res f \frac{dT_1}{T_1} \wedge \cdots \wedge \frac{dT_n}{T_n}
\]
defines a projection of $R_\psi$ onto $K^\flat$ in the category of Banach modules over $K^\flat$; the kernel of this morphism can be expressed as the completed direct sum of $\varphi^n(S)$ over all $s \geq 0$, where
\[
S = \bigoplus{(e_1,\dots,e_n) \in (\ZZ/p\ZZ)^n \setminus (0,\dots,0)} T_1^{e_1} \cdots T_n^{e_n} R_\psi^p.
\]
In particular, on this kernel, $\varphi-1$ vanishes and its cokernel is isomorphic to $S$.

Using the construction from \cite[Lemma~7.1.7]{part2}, we may exhibit a chain homotopy witnessing the vanishing of $H^i_\Gamma(S)$ for all $i \geq 0$; it follows that the morphisms
\[
H^i_{\varphi, \Gamma}(K^\flat) \to H^i_{\varphi, \Gamma}(R_\psi),
\qquad
H^i_{\varphi, \Gamma}(K^{\prime \flat}) \to H^i_{\varphi, \Gamma}(R_{\psi'})
\]
are isomorphisms. Since
\[
\ker(\varphi-1, K^\flat) = \ker(\varphi-1, K^{\prime \flat}) = \FF_p, \qquad
\coker(\varphi-1, K^\flat) = \coker(\varphi-1, K^{\prime \flat}) = 0,
\]
this yields the desired result.
\end{proof}

We now recover Theorem~\ref{T:finiteness1}(a) and a bit more.

\begin{theorem} \label{T:etale cohomology over point}
Assume that $K$ is algebraically closed.
Let $X$ be a smooth proper rigid analytic variety over $K$.
Let $V$ be an \'etale $\Qp$-local system on $X$.
\begin{enumerate}
\item[(a)]
The pro-\'etale cohomology groups
$H^i(X_{\proet}, V)$ are finite-dimensional $\QQ_p$-vector spaces;
moreover, these groups vanish for $i > 2n$.
\item[(b)]
The formation of the groups in (a) commutes with base change from $K$ to a larger algebraically closed nonarchimedean field.
\item[(c)]
Let $f: X \to \Spa(K,K^+)$ be the structure morphism,
and let $\calF$ be the \'etale $(\varphi, \Gamma)$-module
over $\tilde{\bC}^{[s,r]}_X$ associated to $V$.
Then for each $i$, $R^i f_* \calF$ is an \'etale projective $(\varphi, \Gamma)$-module.
\end{enumerate}
\end{theorem}
\begin{proof}
 We consider the category of \emph{Banach-Colmez spaces over $K$} (i.e., the \emph{Espaces de Banach de dimension finie} of \cite{colmez-banach}); roughly speaking, these are objects in the category of topological $\QQ_p$-vector spaces which, in the quotient by the Serre subcategory of finite-dimensional $\QQ_p$-vector space, becomes isomorphic to the restriction of scalars of a finite-dimensional $K$-vector space. 
Each such space has a \emph{$K$-dimension}, which is a nonnegative integer, and a \emph{$\QQ_p$-dimension}, which can be any integer in general but must be nonnegative if the $K$-dimensional vanishes; these are both additive in short exact sequences.

Let $\calC$  denote the category of $\varphi$-modules over the Robba ring associated to $K$ (or equivalently, vector bundles on the Fargues-Fontaine curve associated to $K$. The basic properties of $\calC$ that we need (e.g., see \cite{fargues-fontaine})
are that every object of $\calC$ admits a direct sum into copies of certain standard objects $\calO(s)$ parametrized by $s \in \QQ$, and the cohomology groups of these objects (which are only supported in degrees 0 and 1) are Banach-Colmez spaces over $K$ with the following additional constraints.
\begin{itemize}
\item
For $s<0$, $H^0$ vanishes and $H^1$ has positive $K$-dimension (determined by $s$).
\item
For $s=0$, $H^0$ has $K$-dimension 0 (but nonzero $\QQ_p$-dimension) and $H^1$ vanishes.
\item
For $s>0$, $H^0$ has positive $K$-dimension (determined by $s$) and $H^1$ vanishes.
\end{itemize}
From \eqref{eq:etale direct images derived} and the preceding discussion,
it follows that each group $H^i(Y_{\proet}, V)$ is a Banach-Colmez space; moreover, base extension from $K$ to $K'$ preserves the $K$-dimension. However, by applying Lemma~\ref{L:cohomology base field} locally on $Y$, we see that $R^i f_* V$ is invariant under an algebraically closed base field extension, so the $K$-dimension of $H^i(Y_{\proet}, V)$ must vanish. It follows that $H^i(Y_{\proet}, V)$ is a finite-dimensional $\QQ_p$-vector space; this proves (a) and (b).
Again from \eqref{eq:etale direct images derived} and the classification, it follows that $R^i f_* \calF$ splits as a direct sum of copies of $\calO(s)$ where only $s=0$ is allowed.
This proves (c).
\end{proof}

We now promote the previous statement to the relative case.

\begin{theorem} \label{T:etale direct images type C}
With notation as in Theorem~\ref{T:direct image}, suppose that 
$\calF$ is the \'etale $(\varphi, \Gamma)$-module
over $\tilde{\bC}^{[s,r]}_Y$
associated to the $\QQ_p$-local system $V$ on $Y$.
For each $i$, the map \eqref{eq:etale direct images}
is an isomorphism identifying $R^i f_* V$ with $(R^i f_* \calF)^{\varphi=1}$.
In particular, $R^i f_* V$ is a $\QQ_p$-local system and $R^i f_* \calF$ is an \'etale projective $(\varphi, \Gamma)$-module.
\end{theorem}
\begin{proof}
We proceed by descending induction on $i$. Given the claim for degree $\geq i+1$,
Theorem~\ref{T:general base change} implies that the formation of $R^i f_* \calF$ commutes with arbitrary base change.

Suppose that $X$ is a curve and $K$ is algebraically closed. 
By Remark~\ref{R:coherent type C}, the torsion submodule of $R^i f_* \calF$ is annihilated by some $t \in \tilde{\bC}^{[s,r]}_K$; in particular, any torsion persists upon base extension to a geometric point. By Theorem~\ref{T:etale cohomology over point}, this means that $R^i f_* \calF$ is torsion-free, and in particular co-$\theta$-projective
(see Definition~\ref{D:co-theta-projective}).
Moreover, there exists a Zariski-closed nowhere dense subspace $Z$ of $X$ such $R^i f_* \calF|_{X\setminus Z}$ is projective (namely the support of the Fitting ideal of
$R^i f_* \calF$; see \cite[Remark~8.7.6]{part2}).

By Theorem~\ref{T:pseudocoherent noetherian}, we may form the double dual $\calG$ of $R^i f_* \calF$ in the category of coherent $(\varphi, \Gamma)$-modules; by the previous paragraph, the natural morphism $R^i f_* \calF \to \calG$ is injective. 
By \cite[Theorem~5.9.4, Theorem~7.1.9]{part2}, $\calG$ descends to a reflexive module over a two-dimensional regular noetherian ring, which is therefore projective.

If we pull back to a geometric point of $X \setminus Z$, then $R^i f_* \calF \to \calG$ becomes an isomorphism and both sides are \'etale by Theorem~\ref{T:etale cohomology over point}. Since $Z$ is nowhere dense, this implies that the degree of $\calG$ is everywhere 0;
moreover, the cokernel $\calH$ of $R^i f_* \calF \to \calG$ has generic rank 0, so it must be $\theta$-local.

Let $g: X' \to X$ be a geometric point supported in $Z$. Then $L_1 g^* \calH$ is again killed by some power of $t_\theta$; since
\[
0 \to L_1 g^* \calH \to g^* R^i f_* \calF \to g^* \calG \to g^* \calH \to 0
\]
is exact and $g^* R^i f_* \calF$ is torsion-free (again by Theorem~\ref{T:etale cohomology over point}, we must have $L_1 g^* \calH = 0$. Now $g^* R^i f_* \calF \to g^* \calG$ is an injection with $\theta$-local cokernel, so $\deg(g^* R^i f_* \calF) \leq \deg(g^* \calG)$ with equality if and only if $\calH = 0$. However, $g^* R^i f_* \calF$ is \'etale, so its degree is 0, whereas $g^* \calG$ has degreee $0$ as calculated above. We conclude that the support of $\calH$ is empty, and so $R^i f_* \calF \cong \calG$.

We now drop the restriction on $X$. By the curve case, the fibers of $R^i f_* \calF$ over geometric points of $X$ are projective of locally constant rank; by \cite[Proposition~2.8.4]{part1}, $R^i f_* \calF$ is projective. Since $R^i f_* \calF$ is fiberwise \'etale, \cite[Corollary~7.3.9]{part1} implies that $R^i f_* \calF$ is \'etale, that is,
the map
\[
(R^i f_* \calF)^\varphi=1 \otimes_{\underline{\QQ}_p} \tilde{\bC}^{[s,r]}_X \to 
R^i f_* \calF
\]
is an isomorphism. It follows that $\varphi-1$ is surjective on $R^i f_* \calF$, so 
the fact that \eqref{eq:etale direct images derived} is a quasi-isomorphism now implies that 
\eqref{eq:etale direct images} is an isomorphism, as desired.
\end{proof}

\begin{remark} \label{R:local system v-topology}
Recall that by combining Theorem~\ref{T:Artin-Schreier rational} with the properties of the v-topology described in the proof of Lemma~\ref{L:base change point1}, we may immediately deduce that the categories of $\QQ_p$-local systems on $X$ for the pro-\'etale topology and the v-topology coincide, and the pullback equivalence is moreover exact
(compare \cite[Remark~4.5.2]{part2}.) In light of  Theorem~\ref{T:general base change} and Theorem~\ref{T:etale direct images type C}, we may further deduce that higher direct images of $\QQ_p$-local systems along smooth proper morphisms can be computed using the v-topology in place of the pro-\'etale topology.
In particular, formation of these higher direct images commutes with arbitrary base change on $X$, as expected by analogy with the proper base theorem for \'etale cohomology of schemes.
\end{remark}

The question of when the higher direct images of a projective but not \'etale $(\varphi, \Gamma)$-module are again projective is somewhat subtler. One easy observation is the following.

\begin{prop} \label{P:direct image etale seed}
With notation as in Theorem~\ref{T:direct image}, suppose that 
$\calF$ is a coherent $(\varphi, \Gamma)$-module over $\tilde{\bC}^{[s,r]}_X$.
Suppose also that each connected component of $X$ contains a point $x$ such that:
\begin{itemize}
\item
$\calF$ is projective on some neighborhood of $f^{-1}(x)$; and
\item
$f^{-1}(x)$ is contained in the \'etale locus of $\calF$.
\end{itemize}
Then the higher direct images of $\calF$ are co-$\theta$-projective.
\end{prop}
\begin{proof}
By Remark~\ref{R:coherent type C}, the prime-to-$t_\theta$ torsion of each higher direct 
image has locally constant rank on $X$. By Theorem~\ref{T:general base change}, it thus suffices to rule out the existence of such torsion on one fiber per connected component of $X$; this is immediate from Theorem~\ref{T:etale direct images type C}.
\end{proof}

We next consider the case of a de Rham $(\varphi, \Gamma)$-module.
\begin{defn}
Suppose that $K$ is discretely valued with perfect residue field.
Define the sheaves $\calO \bB^+_{\dR,X}$, $\calO \bB_{\dR,X}$ on $X_{\proet}$
as in \cite[Definition~8.6.5]{part2}. For $\calF$ a $(\varphi, \Gamma)$-module over $\tilde{\bC}^{[s,r]}_X$ with $1 \in [s,r]$,
let $\nu_{\proet}: X_{\proet} \to X$ be the canonical morphism, and define
\[
D_{\dR}(\calF) := \nu_{\proet *} (\calF \otimes_{\tilde{\bC}^{[s,r]}_X} \calO \bB_{\dR,X}).
\]
This is a coherent $\calO_X$-module: namely, by \cite[Theorem~8.6.2]{part2} this holds if $X$ is smooth, and the general case then follows by resolution of singularities.
If $X$ is smooth, we moreover have a canonical $\calO_X$-linear connection on $D_{\dR}(\calF)$ (see \cite[Theorem~7.2]{scholze2}), forcing $D_{\dR}(\calF)$ to be projective.

There is a canonical injective morphism
\[
D_{\dR}(\calF) \otimes_{\calO_X} \calO \bB_{\dR,X} \to \calF \otimes_{\tilde{\bC}^{[s,r]}_X} \calO \bB_{\dR,X}.
\]
We say that $\calF$ is \emph{de Rham} if this morphism is an isomorphism.
In case $\calF$ is the \'etale $(\varphi, \Gamma)$-module associated to a $\QQ_p$-local system $V$, we say that $V$ is de Rham if $\calF$ is;
by a theorem of the second author and X. Zhu \cite{liu-zhu}, this holds if and only if each connected component of $X$ contains a classical point at which $V$ is de Rham in the usual sense of Fontaine.
\end{defn}

\begin{theorem} \label{T:direct image de Rham}
With notation as in Theorem~\ref{T:direct image}, suppose that 
$K$ is discretely valued with perfect residue field,
$X$ is smooth, and $\calF$ is a de Rham $(\varphi, \Gamma)$-module
over $\tilde{\bC}^{[s,r]}_Y$. Then the higher direct images of $\calF$ along $Y_{\proet} \to X_{\proet}$ are de Rham $(\varphi, \Gamma)$-modules over $\tilde{\bC}^{[s,r]}_X$,
and there are natural isomorphisms
\begin{equation} \label{eq:de rham isom}
R^i f_* D_{\dR}(\calF) \cong D_{\dR}(R^i f_{\proet *} \calF).
\end{equation}
\end{theorem}
\begin{proof}
In light of Theorem~\ref{T:direct image}
and the assumption that $K$ is discretely valued, it suffices to check projectivity on the level of $t_\theta$-adic completions,
i.e., to check that \eqref{eq:de rham isom} is an isomorphism.
This follows as in \cite[Theorem~7.11]{scholze2}.
\end{proof}

Combining
Theorem~\ref{T:etale direct images type C} and Theorem~\ref{T:direct image de Rham}
gives a relative version of the \'etale-de Rham comparison isomorphism, generalizing
\cite[Theorem~1.10]{scholze2} which treats the case of $\ZZ_p$-local systems.
\begin{theorem} \label{T:comparison isomorphism}
Let $K$ be a complete discretely valued field of characteristic $0$ whose residue field is perfect of characteristic $p$.
Let $f: Y \to X$ be a smooth proper morphism of rigid analytic spaces over $K$. Let $V$ be a de Rham $\QQ_p$-local system on $Y$, and let $\calF$ be its associated \'etale $(\varphi, \Gamma)$-module.
\begin{enumerate}
\item[(a)]
The higher direct images $R^i f_{\proet *} \calF$ are \'etale de Rham $(\varphi, \Gamma)$-modules.
\item[(b)]
The canonical morphisms $(R^i f_{\proet *} V) \otimes_{\underline{\QQ}_p} \tilde{\bC}_X \to R^i f_{\proet *} \calF$ are isomorphisms.
Equivalently, the canonical morphisms
$R^i f_{\proet *} V \to (R^i f_{\proet *} \calF)^{\varphi=1}$ are isomorphisms.
\item[(c)]
The canonical morphisms $R^i f_* D_{\dR}(\calF) \to D_{\dR}(R^i f_{\proet *} \calF)$
are isomorphisms.
\end{enumerate}
\end{theorem}

We point out a corollary suggested by David Hansen.
\begin{cor}
For $K, f$ as in Theorem~\ref{T:comparison isomorphism}, suppose that $X$ is itself smooth over $K$ (as then is $Y$). Then the sheaves
$R^i f_* \Omega^j_{Y/X}$ on $X$ are vector bundles for all $i,j \geq 0$.
\end{cor}
\begin{proof}
Take $V = \underline{\QQ_p}$ in Theorem~\ref{T:comparison isomorphism}
and identify the sheaves in question with the graded quotients
of the filtered sheaf $D_{\dR}(\calF)$. By \cite[Theorem~3.8(ii)]{liu-zhu},
these graded quotients are locally free.
\end{proof}

\begin{remark}
The point of view implicit in Theorem~\ref{T:comparison isomorphism} is that it is the $(\varphi, \Gamma)$-module $\calF$, rather than the $\QQ_p$-local system $V$, that lies at the heart of the comparison isomorphism. This point of view is shared by the work of Bhatt, Morrow, and Scholze \cite{bms}, in which the crystalline comparison isomorphism is constructed via a certain $\bA_{\mathrm{inf}}$-valued cohomology theory (for more on which see \cite{morrow}).
\end{remark}

\begin{remark}
The restriction to discretely valued fields in the last few results is a side effect of the fact that the de Rham condition on $(\varphi, \Gamma)$-modules requires discreteness in order to make the definition. If one could extend this definition in a sensible way to  more general $K$, one could hope to generalize these results also.
\end{remark}


\begin{thebibliography}{99}

\bibitem{bms}
B. Bhatt, M. Morrow, and P. Scholze,
Integral $p$-adic Hodge theory,
arXiv:1602.03148v1 (2016).

\bibitem{colmez-banach}
P. Colmez,
Espaces de Banach de dimension finie,
\textit{J. Inst. Math. Jussieu} \textbf{1} (2002), 331--439. 

\bibitem{dejong-vanderput}
A.J. de Jong and M. van der Put,
\'Etale cohomology of rigid analytic spaces,
\textit{Doc. Math.} \textbf{1} (1996), 1--56.

\bibitem{fargues-fontaine}
L. Fargues and J.-M. Fontaine, Courbes et fibr\'es vectoriels en th\'eorie de Hodge
$p$-adique, in preparation; draft (September 2015) available at
\url{http://webusers.imj-prg.fr/~laurent.fargues/}.

\bibitem{herr1}
L. Herr, Sur la cohomologie galoisienne des corps $p$-adiques,
\textit{Bull. Soc. Math. France} \textbf{126} (1998), 563--600.

\bibitem{herr2}
L. Herr, Une approche nouvelle de la dualit\'e locale de Tate,
\textit{Math. Ann.} \textbf{320} (2001), 307--337.

\bibitem{huber}
R. Huber,
\textit{\'Etale Cohomology of Rigid Analytic Varieties and Adic Spaces},
Aspects of Mathematics, E30,
Friedr. Vieweg \& Sohn, Braunschweig, 1996.

\bibitem{kedlaya-rigid-finiteness}
K.S. Kedlaya, Finiteness of rigid cohomology with coefficients,
\textit{Duke Math. J.} \textbf{134} (2006), 15--97.

\bibitem{kedlaya-course}
K.S. Kedlaya, \textit{$p$-adic Differential Equations},
Cambridge Studies in Advanced Math.\ 125, Cambridge Univ. Press, Cambridge,
2010.

\bibitem{kedlaya-noetherian}
K.S. Kedlaya, Noetherian properties of Fargues-Fontaine curves, \textit{Int. Math. Res. Notices} (2015), article ID rnv227.

\bibitem{kedlaya-reified}
K.S. Kedlaya, Reified valuations and adic spectra, \textit{Res. Num. Theory}
(2015) 1:20 (41 pages).

\bibitem{part1}
K.S. Kedlaya and R. Liu, Relative $p$-adic Hodge theory: Foundations,
\textit{Ast\'erisque} \textbf{371} (2015), 239 pages; errata,  \cite[Appendix]{part2}.

\bibitem{part2}
K.S. Kedlaya and R. Liu, Relative $p$-adic Hodge theory, II: Imperfect period rings,
arXiv:1602.06899v2 (2016). 

\bibitem{kpx}
K.S. Kedlaya, J. Pottharst, and L. Xiao, Cohomology of arithmetic families of
$(\varphi, \Gamma)$-modules, \textit{J. Amer. Math. Soc.}
\textbf{27} (2014), 1043--1115.

\bibitem{kiehl-finiteness}
R. Kiehl, Der Endlichkeitssatz f\"ur eigentliche Abbildungen in der nichtarchimedische Funktionentheorie, \textit{Invent. Math.} \textbf{2} (1967), 191--214.

\bibitem{liu-herr}
R. Liu, Cohomology and duality for $(\varphi, \Gamma)$-modules over the Robba ring, 
\textit{Int. Math. Res. Notices} \textbf{2007}, article ID rnm150 (32 pages).

\bibitem{liu-zhu}
R. Liu and X. Zhu, Rigidity and a Riemann-Hilbert correspondence for
$p$-adic local systems, arXiv:1602.06282v3 (2016); to appear in
\textit{Invent. Math.}

\bibitem{morrow}
M. Morrow, Notes on the $\mathbb{A}_{\mathrm{inf}}$-cohomology of \textit{Integral $p$-adic Hodge theory}, arXiv:1608.00922v1 (2016).

\bibitem{scholze2}
P. Scholze, $p$-adic Hodge theory for rigid analytic varieties, 
\textit{Forum of Math. Pi} \textbf{1} (2013), doi:10.1017/fmp.2013.1;
erratum available at \url{http://www.math.uni-bonn.de/people/scholze/pAdicHodgeErratum.pdf}.

\bibitem{scholze-berkeley}
P. Scholze, $p$-adic geometry, UC Berkeley course notes, 2014;
notes by Jared Weinstein available at
\url{http://math.berkeley.edu/~jared/Math274/ScholzeLectures.pdf}.

\bibitem{temkin-local}
M. Temkin, On local properties of non-Archimedean analytic spaces,
\textit{Math. Ann.}
\textbf{318} (2000), 585--607.

\bibitem{temkin-local2}
M. Temkin, On local properties of non-archimedean analytic spaces, II, 
\textit{Israel J. Math.}
\textbf{140} (2004), 1--27.

\end{thebibliography}
\end{document}